\documentclass[reqno, 12pt]{amsart}

\usepackage[fontsize=13pt]{scrextend}

\usepackage[a4paper, margin=.75in]{geometry}

\usepackage{amsmath, amsthm, thmtools, amsfonts, amssymb, mathtools}
\usepackage{pdflscape, blkarray, multirow, booktabs}

\usepackage{enumerate}
\usepackage{amstext} 
\usepackage{array}   
\usepackage[dvipsnames]{xcolor}
\usepackage[pdfpagelabels, bookmarks, hyperindex, hyperfigures, colorlinks = true, linkcolor = {RoyalBlue}, citecolor = {BrickRed}, urlcolor = {Magenta}]{hyperref}

\usepackage{tikz,lipsum,lmodern}
\usepackage[most]{tcolorbox}

\usepackage{float}
\usepackage{import}
\usepackage{xifthen}
\usepackage{pdfpages}
\usepackage{transparent}

\usepackage{makecell}

\numberwithin{equation}{section}

\allowdisplaybreaks

\definecolor{citecol}{RGB}{12,127,172}
\definecolor{example}{rgb}{0.6, 0.4, 0.8}
\definecolor{remark}{rgb}{0.0, 0.5, 0.0}

\newcommand{\innerprod}[1]{\ensuremath{\left\langle \right\rangle}}

\DeclareMathOperator{\codim}{codim}

\theoremstyle{definition}
\newtheorem{thm}{Theorem}[section]
\newtheorem{lemma}[thm]{Lemma}
\newtheorem{remark}[thm]{Remark}
\newtheorem{defn}[thm]{Definition}
\newtheorem{prop}[thm]{Proposition}
\newtheorem{corollary}[thm]{Corollary}

\newtheorem{obs}[thm]{Observation}

\newtheorem{thmA}{Theorem}

\begin{document}
\allowdisplaybreaks
\title[Cut Locus of Finsler Submanifold]{On the Cut Locus of Submanifolds of a Finsler Manifold}

\author[A. Bhowmick]{Aritra Bhowmick}
\address{Department of Mathematics and Statistics, IISER Kolkata \\ West Bengal, India}
\email{avowmix@gmail.com}

\author[S. Prasad]{Sachchidanand Prasad}
\address{ICTS-TIFR Bangalore, Karnataka, India}
\email{sachchidanand.prasad1729@gmail.com}



\begin{abstract}
    In this article, we investigate the cut locus of closed (not necessarily compact) submanifolds in a forward complete Finsler manifold. We explore the deformation and characterization of the cut locus, extending the results of Basu and the second author (\emph{Algebraic and Geometric Topology}, 2023). Given a submanifold $N$, we consider an $N$-geodesic loop  as an $N$-geodesic starting and ending in $N$, possibly at different points. This class of geodesics were studied by Omori (\emph{Journal of Differential Geometry}, 1968). We obtain a generalization of Klingenberg's lemma for closed geodesics (\emph{Annals of Mathematics}, 1959) for $N$-geodesic loops in the reversible Finsler setting.
 
\end{abstract}

\date{\today}
\maketitle

\setcounter{tocdepth}{3}

\frenchspacing 

\section{Introduction} \label{sec:introduction}
In the field of Riemannian geometry, the study of geodesics plays a fundamental role in comprehending the geometric properties of the underlying space. \emph{Geodesics} are (locally) the shortest paths between two points on a Riemannian manifold, akin to straight lines in Euclidean spaces. One important concept related to geodesics is the \emph{cut locus}. The cut locus of a point (or a submanifold) of a Riemannian manifold is the set of points beyond which a distance minimizing geodesic from the point (or the submanifold) ceases to be distance minimal. Cut locus of a point, a notion introduced by Henri Poincar\'{e} \cite{Poin05}, has been extensively studied (see \cite{Kob67} for a survey as well as \cite{Buc77},  \cite{Mye35}, \cite{Wol79}, and \cite{Sak96}). Prior to Poincar\'{e}, its implicit appearance may be traced back to a paper by Mangoldt \cite{Man81}. Other articles such as \cite{Whi35} and \cite{Mye35, Mye36} describe the topological properties of the cut locus. The exploration of the cut locus of a Riemannian manifold plays a crucial role in understanding the manifold's underlying geometry, leading to many interesting results in Riemannian geometry. For instance, Rauch's proof of the “Sphere theorem”  \cite{Rau59} heavily relies on estimates of distances to the cut locus. Additionally, it was realized that much of a manifold's topological properties lie in its cut locus. A very good account of these results and methods can be found in the articles \cite{Kli59, Kob67, Wein68}.

Finsler manifolds provide a generalization of Riemannian manifolds by considering a broader class of metrics that capture both the length and the direction of tangent vectors. The study of Finsler geometry has yielded significant insights and applications in various fields, including physics, optimization theory, and robotics. Essentially, a \emph{Finsler manifold} is a manifold $M$ where each tangent space is equipped with a Minkowski norm, which is a norm not necessarily induced by an inner product. In sharp contrast to the Riemannian case, this norm induces canonical inner products that are not parameterized by points of $M$, but by directions in $TM$. Thus, one can think of a Finsler manifold as a space where the inner product does not only depend on where you are but also in which direction you are looking. While a Riemannian manifold is itself trivially a Finsler manifold, where the inner product at a point is the same in all the tangent directions, Finsler geometry represents quite a large step away from the Riemannian realm. Nonetheless, Finsler geometry contains analogs for many of the natural objects in Riemannian geometry. For example, the length of a curve, geodesics, curvature, connections, and covariant derivatives all generalize; however, normal coordinates do not \cite{Rund59}. In this article, we focus on investigating the cut locus of a submanifold within a Finsler manifold. For a Finsler manifold, the cut locus of a point has been extensively studied, see \cite{Warner1965, Wein68, Hass73, SaTa16}. 

The second author with Basu in \cite{BaPr21} showed that the complement of the cut locus of a submanifold in a Riemannian manifold topologically deforms to $N$. In this article, we prove the same for Finsler manifolds as well. Let $(M, F)$ be a Finsler manifold $M$ with a forward complete Finsler metric $F$, and $N$ be a closed submanifold of $M$. If $N$ is noncompact, we additionally assume that $F$ is backward complete.

\begin{thmA}[\autoref{thm:deformation}]
    If $\mathrm{Cu}(N) $ denotes the cut locus of $N$, then
    \begin{enumerate}
        \item $M\setminus \mathrm{Cu}(N)$ strongly deformation retracts onto $N$, and
        \item $M \setminus N$ strongly deformation retracts onto $\mathrm{Cu}(N)$.
    \end{enumerate}  
\end{thmA}

Along with the above result, we give two characterizations of the cut locus. The first one is in terms of the focal locus (\autoref{thm:cutPointClassification}), and the second one is in terms of separating set (\autoref{thm:separatingSetDense}). Recall that the \emph{focal locus} is the collection of points where the normal exponential map fails to be of full rank (\autoref{defn:focalLocus}), whereas the \emph{separating set} of a submanifold is the collection of all points which can be joined by more than one minimizing geodesic from the submanifold (\autoref{defn:separtingSet}). The second characterization is particularly important in calculations, as shown in \cite{Pra23}. Let $\mathrm{Se}(N)$ denote the separating set of $N$.

\begin{thmA}[\autoref{thm:cutPointClassification}]
    Let $N$ be a closed submanifold of a Finsler manifold $(M, F)$. Then, $q \in \mathrm{Cu}(N)$ if and only if $\gamma : [0, \ell] \to M$ is an $N$-segment with $q = \gamma(\ell)$ and at least one of the following holds true.
    \begin{enumerate}
        \item There exist two distinct $N$-segments joining $N$ to $q$, i.e., $q \in \mathrm{Se}(N)$, or
        \item $q$ is the first focal point of $N$ along $\gamma$.
    \end{enumerate}
\end{thmA}

\begin{thmA}[\autoref{thm:separatingSetDense}]
    Let $N$ be a closed submanifold of a Finsler manifold $(M, F)$. Then, $\mathrm{Cu}(N) = \overline{ \mathrm{Se}(N) }$.
\end{thmA}

In Finsler manifolds, geometric objects possess a two-sided nature, exhibiting both forward and backward characteristics. This duality arises due to the inherent asymmetry of the distance function in Finsler geometry. Consequently, Klingenberg's celebrated result on closed geodesics loops in Riemannian geometry \cite{Kli59}, does not generalize directly to the Finsler setup, and one needs to look at \emph{reversible} Finsler metrics, which are symmetric in nature. For submanifolds, geodesics emanating from them and also ending in them, possibly at a point different from the origin, have been studied by Omori in \cite{Omo68}. We call them $N$-geodesic loops (\autoref{defn:NGeodesicLoop}). Generalizing the results from \cite{Innami2012} and \cite{Xu2015}, we have the following, which is new even for the Riemannian setup.

\begin{thmA}[\autoref{thm:NGeodesicLoop}]
    Let $N$ be a closed submanifold of a reversible Finsler manifold $(M, F)$. Suppose $x_0 \in \mathrm{Cu}(N)$ is a global minima of the distance function from $N$ on $\mathrm{Cu}(N)$. Then, either $x_0$ is a focal point of $N$, or we have a unique $N$-geodesic loop, crossing the cut locus exactly at $x_0$, at the midpoint.
\end{thmA}

The above result follows from \autoref{lemma:twoNSegments}, which provides a criterion for the existence of exactly two distance minimizing geodesics from a manifold to a cut point and might be of independent interest. We also prove the above theorem only under the assumption that $x_0$ is a \emph{local} minima of the distance function, provided $\mathrm{Cu}(N)$ is compact. As a corollary (\autoref{cor:sphereTheorem}), we obtain a sphere theorem similar to \cite[Lemma 5.2]{InnItoNagShi19}.

\subsection*{Conventions} Throughout this article, all manifolds are assumed to be smooth and without boundary (not necessarily compact). A closed submanifold $N \subset M$ means that $N$ is an embedded submanifold, which is closed in the subspace topology. A Finsler metric $F$ on $M$ is always assumed to be \emph{forward} complete, and not necessarily reversible, unless mentioned otherwise. If $N$ is noncompact, we shall additionally assume that $F$ is backward complete (see \autoref{rmk:distanceNotAchieved}).

\subsection*{Organization of the paper} In \autoref{sec:preliminaries}, we begin by recalling basic definitions from Finsler geometry, including the Finsler metric and the Chern connection. In the next section (\autoref{sec:submanifolds}), we introduce the cut locus and the focal locus of a submanifold, and derive some basic properties. In \autoref{sec:mainResults}, we present our main results including the characterization of the cut locus and the deformation of its complement. In the last part (\autoref{subsec:NGeodesicLoop}) of this section, we address the generalization of Klingenberg's lemma to the Finsler setup.
\section{Preliminaries on Finsler Geometry} \label{sec:preliminaries}
  This section recalls some basic notions about Finsler geometry and cut locus. Primary references for this section are \cite{Shen01, Peter06, Javaloyes2015, Ohta2021}. Throughout this article, we shall use the notation $\widehat{TM} = TM \setminus 0$ to denote the slit tangent bundle of a manifold $M$. Given a bundle $E$ over a manifold, $\Gamma E$ will be considered as the sheaf of sections of $E$, in particular, $X \in \Gamma E$ denotes a \emph{local} section of $E$ with some unspecified open set in $M$ as the domain of definition.
\subsection{Finsler Metric} Finsler metric was first studied by P. Finsler in his dissertation \cite{Fin51}. In his thesis, he studied the variational problems of Finsler metrics. 
\begin{defn}\label{defn:FinslerMetric}
    Let $M$ be a smooth manifold, and $TM$ denotes its tangent bundle. A \textit{Finsler metric} on $M$ is a continuous function $F: TM \to \mathbb{R}$ satisfying the following properties.
    \begin{enumerate}
        \item $F$ is smooth on $\widehat{TM}$;
        \item for any $p\in TM$, the restriction $F_p\coloneqq F\big|_{T_pM}$ is a Minkowski norm, i.e.,
        \begin{enumerate}[(a)]
            \item (Positive $1$-homogeneity) for any $\lambda>0$ and $\mathbf{v}\in T_pM\setminus\{0\}$, we have $F_p(\lambda \mathbf{v})=\lambda F_p(\mathbf{v})$, and 
            \item (Strong Convexity) for all $\mathbf{v}\in T_pM\setminus\{0\}$, the symmetric tensor $g_{\mathbf{v}}$ on $T_pM$, called the \emph{fundamental tensor}, is positive definite, where 
            \begin{displaymath}
                g_{\mathbf{v}}(\mathbf{v}_1,\mathbf{v}_2)\coloneqq \left.\dfrac{1}{2} \dfrac{\partial^2}{\partial s_1 \partial s_2} \right|_{s_1=s_2=0} \left(F_p(\mathbf{v} + s_1\mathbf{v}_1 + s_2\mathbf{v}_2)\right)^2.
            \end{displaymath}
        \end{enumerate}
    \end{enumerate}
    $F$ is called a \emph{pseudo-Finsler} (or, \emph{semi-Finsler}) metric if the fundamental tensor is only nondegenerate. $F$ is \emph{reversible} if $F(-\mathbf{v}) = F(\mathbf{v})$ holds for all $\mathbf{v} \in \widehat{TM}$.
\end{defn}

\begin{remark}\label{rmk:pseudoFinsler}
    Although many of the results in \autoref{sec:preliminaries} and \autoref{sec:submanifolds} remain valid for pseudo-Finsler manifolds, our main results do not directly generalize. Crucially, Hopf-Rinow theorem does not hold for the pseudo-Finsler (or even the pseudo-Riemannian) manifolds, and unit tangent spheres need not be compact in these cases. Thus, we only focus on Finsler manifolds.
\end{remark}

As a consequence of the homogeneity, the fundamental tensor satisfies the useful identity 
\[g_\mathbf{v}(\mathbf{v},\mathbf{v}) = F_p(\mathbf{v})^2, \quad \mathbf{v} \in T_p M \setminus 0.\]
For $\mathbf{v} \in T_p M \setminus 0$, the associated \emph{Cartan tensor} on $T_p M$ is a symmetric $3$-tensor defined as
\[C_{\mathbf{v}}(\mathbf{v}_1,\mathbf{v}_2,\mathbf{v}_3) \coloneqq \left. \frac{1}{4} \frac{\partial^3}{\partial s_1 \partial s_2 \partial s_3} \right|_{s_1=s_2=s_3=0} \left( F_p(\mathbf{v} + s_1 \mathbf{v}_1 + s_2 \mathbf{v}_2 + s_3 \mathbf{v}_3) \right)^2.\]
For each $\mathbf{v}\in T_p M \setminus 0$ and $\mathbf{u}, \mathbf{w}\in T_p M$, we have
\[C_{\mathbf{v}}(\mathbf{v}, \mathbf{u},\mathbf{w}) = C_{\mathbf{v}}(\mathbf{u}, \mathbf{v}, \mathbf{w}) = C_{\mathbf{v}}(\mathbf{u}, \mathbf{w}, \mathbf{v}) = 0.\]
We extend the definition of the fundamental tensor to vector fields as follows: \[g_V(X, Y)(p) \coloneqq g_{V_p}(X_p, Y_p),\]
for any $V\in \Gamma \widehat{TM}$ and $X,Y \in \Gamma TM$ defined near $p \in M$. Similarly, we extend the definition of the Cartan tensor as well.

\subsubsection{Chern Connection} Unlike the Levi-Civita connection in the Riemannian context, we have a family of torsionless connections on a Finsler manifold.

\begin{defn}\label{defn:chernConnection}\cite{Rademacher2004, Javaloyes2014}
    For each $V \in \Gamma \widehat{TM}$, we have a unique affine connection \[\nabla^V: \Gamma TM \otimes \Gamma TM \rightarrow \Gamma TM,\] called the \emph{Chern connection}, satisfying the following conditions for any $X, Y, Z \in \Gamma TM$.
    \begin{itemize}
        \item (Torsion freeness) $\nabla^V_X Y - \nabla^V_Y X = [X, Y]$.
        \item (Almost metric compatibility) $X (g_V(Y,Z)) = g_V(\nabla^V_X Y, Z) + g_V(Y, \nabla^V_X Z) + 2 C_V(\nabla^V_X V, Y, Z)$.
    \end{itemize}
    The associated \emph{curvature tensor} is given by
    \[R^V(X,Y) Z \coloneqq \nabla^V_X \nabla^V_Y Z - \nabla^V_Y \nabla^V_X Z - \nabla^V_{[X,Y]} Z, \quad X,Y,Z \in \Gamma TM,\]
    which is skew-symmetric: $R^V(X,Y) = -R^V(Y,X)$.
\end{defn}

For the closely related notion of \emph{Chern curvature}, we refer to \cite{Javaloyes2020,Javaloyes2021}. Using the Chern connection, we can define a covariant derivative along a curve as well.

\begin{defn}\label{defn:covariantDerivative}\cite{Javaloyes2014}
    Given a curve $\gamma : [a,b] \rightarrow  M$ and $W \in \Gamma \gamma^*\widehat{TM}$, the \emph{covariant derivative} along $\gamma$ is defined as \[D^W_\gamma : \Gamma \gamma^* TM \rightarrow \Gamma \gamma^* TM,\]
    which satisfies the following.
    \begin{itemize}
        \item (Linearity) For any $X, Y \in \Gamma \gamma^*TM$ and scalars $\alpha, \beta \in \mathbb{R}$ we have 
        \[D^W_\gamma(\alpha X + \beta Y) = \alpha D^W_\gamma X + \beta D^W_\gamma Y.\]

        \item (Leibniz rule) For any $X \in \Gamma \gamma^*TM$ and a smooth function $f : [a,b] \rightarrow \mathbb{R}$, we have 
        \[D^W_\gamma(f X) = \frac{df}{dt} X + f D^W_\gamma X.\]
        
        \item (Almost metric compatibility) For any $X, Y \in \Gamma \gamma^* TM$, we have \[\frac{d}{dt}g_W(X,Y) = g_W \left( D^W_{\gamma} X, Y \right) + g_W \left( X, D^W_\gamma Y \right) + 2 C_W \left( D^W_\gamma W, X, Y \right).\]
    \end{itemize}
     If $\dot\gamma(t) \ne 0$, then a vector field $X \in \Gamma \gamma^* TM$ is said to be \emph{parallel} with respect to $\gamma$ if $D^{\dot \gamma}_{\gamma} X = 0$.
\end{defn}

\subsubsection{Legendre Transformation} We have a \emph{nonlinear} fiber preserving map $\mathfrak{L} : \widehat{TM} \rightarrow  \widehat{T^*M}$ given by
\[\mathfrak{L}(\mathbf{v})(\mathbf{w}) \coloneqq g_{\mathbf{v}}(\mathbf{v}, \mathbf{w}), \quad \mathbf{v}\in T_p M \setminus 0, \; \mathbf{w} \in T_p M.\]
$\mathfrak{L}$ is called the \emph{Legendre transformation} associated with $(M,F)$. It follows that $\mathfrak{L}$ is a $C^\infty$-diffeomorphism. We extend $\mathfrak{L}$ to all of $TM$ by setting $\mathfrak{L}(0) = 0$. The extension $\mathfrak{L}: TM \rightarrow T^*M$ is then a fiber-preserving \emph{homeo}morphism.

\subsection{Geodesics and Jacobi Fields}
Let us denote the space of piecewise smooth paths $\gamma : [a,b] \rightarrow M$ as $\mathcal{P} = \mathcal{P}([a,b])$. We have two functionals defined on this path space.
\begin{itemize}
    \item (Length Functional) \begin{align*}
        L : \mathcal{P} &\rightarrow \mathbb{R} \\
        \gamma &\mapsto \int_a^b F(\dot \gamma(t)) dt
    \end{align*}

    \item (Energy Functional) \begin{align*}
        E : \mathcal{P} &\rightarrow \mathbb{R} \\
        \gamma &\mapsto \frac{1}{2}\int_a^b F(\dot \gamma(t))^2 dt
    \end{align*}
\end{itemize}
Note that $L$ is independent of the reparametrization of a curve, whereas $E$ is not. We have 
\[L(\gamma) \le \sqrt{(b-a) E(\gamma)},\]
where equality holds if and only if $\gamma$ is of constant speed. A geodesic is defined as a critical point of the energy functional in the variational sense.

Let us recall the notion of variation of a given curve. Given $\gamma \in \mathcal{P}$, we identify the tangent space to $\mathcal{P}$ at $\gamma$ as the infinite dimensional vector space of vector fields along $\gamma$: \[T_{\gamma} \mathcal{P} \coloneqq \Gamma \gamma^* TM.\] Given a vector field $W \in \Gamma \gamma^* TM$, a \emph{$W$-variation} of $\gamma$ is a piecewise smooth map $\Lambda : (-\epsilon, \epsilon) \times [a,b] \rightarrow M$ satisfying \[\Lambda(0,t) = \gamma(t), \quad \left.\frac{\partial}{\partial s}\right|_{s=0} \Lambda(s, t) = W(t),\]
for all $(s,t)$ in the domain. We tacitly assume that each curve $t \mapsto \Lambda_s(t) = \Lambda(s,t)$ has the same breakpoints as $\gamma$. We refer to \cite[Lemma 2.1, pg. 81]{Sak96} for the proof of the existence of $W$-variations.

\begin{defn}\label{defn:properVariation}
    A variation $\Lambda$ of a curve $\gamma$ joining $p = \gamma(a)$ to $q = \gamma(b)$ is said to be \emph{proper} if the endpoints are kept fixed, i.e.,
    \[\Lambda(s,a) = p, \quad \Lambda(s,b) = q,\qquad -\epsilon < s < \epsilon.\]
    Similarly, a vector field $W \in \Gamma \gamma^*TM$ is said to be \emph{proper} if $W(a) = 0 = W(b)$ holds.
\end{defn}

We refer to \cite[Chapter 7]{Ohta2021} for the first and second variation formulas for the Length functional (with respect to proper variations), and as for the energy functional we refer to \cite{Javaloyes2015}.

\begin{defn}\label{defn:geodesic}
    A piecewise $C^1$ curve $\gamma : [a,b] \rightarrow  M$ is called a \emph{geodesic} if it is a critical point of the energy functional with respect to proper variations, i.e., for any proper variation $\Lambda$ of $\gamma$ we have $\left. \frac{d}{ds} \right|_{s=0} E(\Lambda_s) = 0$.
\end{defn}

As usual, a geodesic is locally distance minimizing, where we have the following notion of Finsler distance.
\begin{defn}\label{defn:distanceMetric}
    Given a Finsler manifold $(M,F)$, the \emph{Finsler distance} between points $p, q \in M$ is defined as
    \[d(p, q) = \inf \left\{ L(\gamma) \;\middle|\; \gamma \in \mathcal{P}, \gamma(a) = p, \gamma(b) = q \right\}.\]
\end{defn}

\begin{remark}\label{rmk:pseudometric}
    In general, the distance function is asymmetric, although it induces the same topology as that of the manifold. If $F$ is reversible, then $d$ is a symmetric metric (see, e.g., \cite[pg. 151]{Bao2000}).
\end{remark}

In general, geodesics fail to be globally distance minimizing.
\begin{defn}\label{defn:minimizer}
    A geodesic $\gamma : [0,\ell] \rightarrow  M$ is said to be a \emph{global distance minimizer} (or simply a \emph{minimizer}) if $d(\gamma(0), \gamma(\ell)) = \ell = L(\gamma)$.
\end{defn}

\begin{prop}\label{prop:geodesicEquivalentDefn}
    Given a unit-speed curve $\gamma$ (i.e., $F(\dot\gamma(t)) = 1$ for all $t$), the following are equivalent.
    \begin{itemize}
        \item $\gamma$ is a geodesic.
        \item $\gamma$ is a critical point of the length functional with respect to proper variations.
        \item $\gamma$ is locally distance minimizing, i.e., for any $a < t_0 < b$, there exists an open interval containing $t_0$, say, $t_0 \in (a_{t_0}, b_{t_0}) \subset [a,b]$, such that $d(\gamma(s), \gamma(t)) = t - s$ holds for any $a_{t_0} \le s \le t \le b_{t_0}$.
        \item $\dot\gamma$ is a parallel vector field with respect to $\gamma$, i.e., $D^{\dot \gamma}_{\gamma} \dot \gamma = 0$.
    \end{itemize}
\end{prop}

For a proof of the above equivalences, we refer to \cite[Chapter 7]{Ohta2021}. Geodesics are precisely the solutions to an initial value problem known as the geodesic equation (e.g., \cite[Equation 3.15]{Ohta2021}). As such, geodesics are always smooth, and given a vector $\mathbf{v} \in T_p M$, we have a unique maximal geodesic $\gamma_{\mathbf{v}} : [0,\ell] \rightarrow  M$ satisfying \[\gamma_{\mathbf{v}}(0) = p, \qquad \dot \gamma_{\mathbf{v}}(0) = \mathbf{v}.\]
Note that, unlike Riemannian geometry, due to the asymmetry of the Finsler metric, the reversed curve $\bar{\gamma}:[0, \ell] \rightarrow M$ defined by $\bar{\gamma}(t) = \gamma(\ell - t)$ need not be a geodesic. But $\bar{\gamma}$ is a geodesic with initial velocity $-\dot\gamma(\ell)$ for the reverse Finsler metric $\bar{F}$ defined by $\bar{F}(\mathbf{v}) = F(-\mathbf{v})$ for all $\mathbf{v} \in TM$ \cite[Section 2.5]{Ohta2021}. Consequently, if two distinct unit-speed geodesics intersect, they do so with distinct velocities.

A Finsler manifold $(M, F)$ is said to be \emph{forward complete} if for all $\mathbf{v} \in TM$, the geodesic $\gamma_{\mathbf{v}}$ is defined for all time $[0, \infty)$. We say $(M, F)$ is \emph{backward complete} if $(M, \bar{F})$ is forward complete. By the Hopf-Rinow theorem \cite[Theorem 3.21]{Ohta2021}, if a Finsler manifold $(M, F)$ is forward or backward complete, then given any two points $p, q \in M$ there exists a minimizer with respect to $F$, and another (possibly distinct) minimizer with respect to $\bar{F}$ joining $p$ to $q$. As mentioned in the introduction, throughout this article, we shall assume $(M, F)$ to be a forward complete Finsler manifold, unless otherwise stated.

\begin{defn}\label{defn:exponentialMap}
    The \emph{exponential map} \[\exp : TM \rightarrow  M\] at a point $p \in M$ is defined as $\exp_p(\mathbf{v}) = \gamma_{\mathbf{v}}(1)$ for any $\mathbf{v} \in T_p M$, where $\gamma_{\mathbf{v}} : [0,\infty) \rightarrow M$ is the unique geodesic starting at $p$ in the direction of $\mathbf{v}$.
\end{defn}

It follows from the theory of ordinary differential equation that the exponential map is smooth on $\widehat{TM}$, but only $C^1$ on the whole tangent bundle. In fact, the exponential is $C^2$ on $TM$ precisely when the Finsler metric comes from a Riemannian metric, in which case the exponential map is $C^\infty$ on $TM$ as well. We refer to \cite[Section 5.3]{Bao2000} for details.

We have the following easy observation.
\begin{prop}\label{prop:globalMinimizerIsLocalMinimizer}
    If $\gamma : [0, \ell] \rightarrow  M$ is a minimizer, then for any $0 \le a \le b \le \ell$ we have $\gamma|_{[a,b]}$ is a minimizer, and in particular $d(\gamma(a), \gamma(b)) = b - a$.
\end{prop}

Any piecewise smooth curve $\gamma$ joining $p$ to $q$, and satisfying $L(\gamma) = d(p, q)$ is, up to a reparametrization, a minimizing geodesic from $p$ to $q$, and in particular $\gamma$ is smooth. A consequence of the completeness assumption is that any two points $p, q \in M$ can be joined by a minimizer, say, $\gamma : [0, \ell] \rightarrow  M$ so that \[\gamma(0) = p, \qquad  \gamma(\ell) = q, \qquad  d(p,q) = \ell = L(\gamma).\]

\subsubsection{Jacobi Fields}
\begin{defn}\label{defn:geodesicVariation}
    Given a unit-speed geodesic $\gamma : [a,b] \rightarrow M$, a variation $\Lambda : (-\epsilon, \epsilon) \times [a,b] \rightarrow M$ of $\gamma$ is said to be a \emph{geodesic variation}, if for each $-\epsilon < s < \epsilon$, the curve $\Lambda_s : [a,b] \rightarrow M$ is a constant-speed geodesic.
\end{defn}

\begin{defn}\label{defn:JacobiField}
    Let $\gamma : [a,b] \rightarrow M$ be a unit-speed geodesic. A vector field $J \in \Gamma \gamma^* TM$ along $\gamma$ is called a \emph{Jacobi field} of $\gamma$ if there exists a geodesic variation $\Lambda : (-\epsilon, \epsilon) \times [a,b] \rightarrow M$ of $\gamma$ satisfying 
    \[\left. \frac{\partial}{\partial s} \right|_{s=0} \Lambda(s,t) = J(t), \quad t\in [a,b].\]
\end{defn}

In order to characterize Jacobi fields via the usual Jacobi equation, we need to define a certain curvature tensor along a geodesic. In general, the curvature tensor $R^V$ at a point $p$ depends on the value of the vector field $V$ near $p$. Now, given a geodesic $\gamma : [a, b] \rightarrow M$, and for any $U, W \in \Gamma \gamma^*TM$, one can define a tensor field $R^\gamma (\dot\gamma, U) W$ along $\gamma$ as follows (see \cite{Javaloyes2014,Javaloyes2014_Correction}). Firstly one defines the \emph{vertical derivative} of $\nabla$ as 
\[P_V(X, Y, Z) = \left. \frac{\partial}{\partial t}\right|_{t=0} \nabla^{V + t Z}_X Y, \quad X, Y, Z \in \Gamma TM, \; V \in \Gamma \widehat{TM}.\]
Then, following \cite[Theorem 2.2]{Javaloyes2014_Correction} one defines 
\[R^\gamma\left( \dot \gamma, U \right) W = \left( R^V \left( V, \tilde{U} \right) \tilde{W} + P_V \left( V, \tilde{W}, [\tilde{U}, V] \right) \right) \circ \gamma, \quad U, W \in \Gamma \gamma^*TM,\]
where $V, \tilde{U}, \tilde{W}$ are arbitrary extensions of $\dot \gamma, U, W$ respectively.

\begin{prop}\label{prop:JacobiEquation}
    Given a unit-speed geodesic $\gamma : [a,b] \rightarrow M$, a vector field $J \in \Gamma \gamma^*TM$ is a Jacobi field if and only if it satisfies the Jacobi equation 
    \[D^{\dot\gamma}_\gamma D^{\dot\gamma}_\gamma J - R^{\gamma}(\dot\gamma, J)\dot\gamma = 0.\]
\end{prop}

Since the Jacobi equation is a second order ODE, given the initial data, $\mathbf{u}, \mathbf{v} \in T_{\gamma(a)}M$, there exists a unique Jacobi field $J$ along $\gamma$ satisfying, $J(a) = \mathbf{u}$ and $D^{\dot\gamma}_\gamma J(a) = \mathbf{v}$. In particular, the collection of all Jacobi fields along $\gamma$ forms a vector space of dimension $2 \dim M$.

\begin{prop}\label{prop:conjugateLocusJacobiField}
    Let $\mathbf{v}\in T_p M$ with $F(\mathbf{v}) = 1$, and $\gamma(t) = \exp_p(t\mathbf{v})$ be the geodesic in the direction of $\mathbf{v}$. Then, the kernel of the map 
    \[d(\exp_p)|_{t\mathbf{v}} : T_{t\mathbf{v}}(T_p M) \left( \cong T_p M \right)  \rightarrow T_{\gamma(t)} M\]
    is naturally isomorphic to the vector space 
    \[\left\{ J \;\middle|\; \text{$J$ is a Jacobi field along $\gamma$, satisfying $J(0) = 0 = J(t)$.} \right\}.\]
\end{prop}

\section{Submanifolds in Finsler Manifolds}\label{sec:submanifolds}
In this section, we generalize the notions from the previous section for submanifolds in a Finsler manifold.
\subsection{Normal Cone Bundles}
Given a submanifold $N$, it is natural to ask for a notion of a normal bundle to it.
\begin{defn}\label{defn:normalCone}
    Given a submanifold $N \subset M$, the set $$\nu_p = \nu_p(N) = \big\{ \mathbf{v} \in T_p M \setminus \left\{ \mathbf{0} \right\} \;\big|\; g_{\mathbf{v}}(\mathbf{v},\mathbf{w}) = 0 \; \forall \mathbf{w} \in T_p N\big\} \cup \left\{ \mathbf{0} \right\}$$
    is called the \emph{normal cone} of $N$ at $p \in N$. The set $\nu = \nu(N) = \cup_{p \in N} \nu_p(N)$ is called the \emph{normal cone bundle} of $N$.
\end{defn}

It should be noted that $\nu(N)$ is \emph{not} a vector bundle in general; in fact, it is a cone bundle. Nevertheless, this is the obvious replacement for the notion of a normal bundle as in the Riemannian setting. We shall denote $\hat{\nu}_p \coloneqq \nu_p \setminus \{ 0 \}$, and $\hat{\nu} \coloneqq \cup_{p\in N} \hat{\nu}_p$ is then the \emph{slit cone bundle}. The Legendre transformation $\mathfrak{L}$ maps $\nu$ homeomorphically onto the annihilator bundle $\nu^* \subset T^*M|_N$ of $TN$ defined as 
\[\nu^*|_p \coloneqq \left\{ \omega \in T_p^*M \;\middle|\; \omega|_{T_p N} = 0 \right\}, \quad p \in N.\] It follows that $\hat{\nu}$ is a smooth submanifold of $TM$ of dimension $\dim N + \codim N = \dim M$, whereas $\nu$ is only a topological submanifold. Note that $\nu$ is a closed subset of $TM$.

\begin{defn}\label{defn:unitConeBundle}
    The \emph{unit cone bundle} of $N$ is denoted as $S(\nu) = \cup_{p \in N} S(\nu_p)$, where \[S(\nu_p) \coloneqq \left\{ \mathbf{v} \in \nu_p \;\middle|\; F_p(\mathbf{v}) = 1 \right\}.\]
\end{defn}

\begin{remark}\label{rmk:unitConeIsCompact}
    If $N$ is a compact submanifold, then $S(\nu)$ is again a compact submanifold of $TM$ of dimension $\dim M - 1$.
\end{remark}

\begin{defn}\label{defn:normalExponentialMap}
    Given $N \subset M$, the \emph{normal exponential map} \[\exp^\nu: \nu(N) \rightarrow  M\] is defined as the restriction of the exponential map to the cone bundle $\nu(N)$.
\end{defn}

As $\hat{\nu}$ is a smooth submanifold of $TM$, the restricted map $\exp^\nu|_{\hat{\nu}}$ is smooth.

\subsection{\texorpdfstring{$N$}{N}-Geodesics and \texorpdfstring{$N$}{N}-Jacobi Fields}
Given a submanifold $N \subset M$, we have the subspace of piecewise smooth paths \[\mathcal{P}_N = \mathcal{P}_N([a,b]) = \{ \gamma \in \mathcal{P}([a,b]) \;|\; \gamma(a) \in N\}\] starting at $N$. For any $\gamma\in\mathcal{P}_N$ we say $\gamma$ is a curve joining $N$ to $\gamma(b)$.
\begin{defn}\label{defn:distanceFromN}
    Given $q \in M$, the distance from $N$ to $q$ is defined as
    \[d(N,q) \coloneqq \inf \left\{ L(\gamma) \;\middle|\; \gamma\in\mathcal{P}_N, \; \gamma(b) = q \right\}.\]
\end{defn}

We identify the tangent space of $\mathcal{P}_N$ at a given $\gamma\in \mathcal{P}_N$ as the infinite dimensional vector space 
\[T_\gamma \mathcal{P}_N \coloneqq \left\{ W \in \Gamma \gamma^*TM \;\middle|\; W(a) \in T_{\gamma(a)} N \right\}.\]
Given $W \in T_\gamma\mathcal{P}_N$, a \emph{$W$-variation} is a piecewise smooth map $\Lambda : (-\epsilon,\epsilon) \times [a,b] \rightarrow  M$ satisfying
\[\Lambda(0, t) = \gamma(t), \quad  \left. \frac{\partial}{\partial s} \right|_{s=0} \Lambda(s, t) = W(t), \quad  \Lambda(s, a) \in N,\]
for all $(s,t)$ in the domain of the definition. The variation is said to be \emph{proper} if it further satisfies $\Lambda(s,b) = \gamma(b)$ for all $-\epsilon < s < \epsilon$. Clearly, proper variations correspond to those $W \in T_\gamma \mathcal{P}_N$ that satisfy $W(b) = 0$. This leads to the first and second variational formulas for the restricted length and energy functionals: $L|_{\mathcal{P}_N}$ and $E|_{\mathcal{P}_N}$. We refer to \cite{Javaloyes2015} for the relevant formulas.

\begin{defn}\label{defn:NGeodesic}
    A piecewise $C^1$ curve $\gamma : [a,b] \rightarrow M$ in $\mathcal{P}(N)$ is called an \emph{$N$-geodesic} if $\gamma$ is a critical point of the restricted energy functional $E|_{\mathcal{P}(N)}$ with respect to proper variations, i.e., for any proper variation $\Lambda$ of $\gamma$ we have $\left. \frac{d}{ds} \right|_{s=0} E(\Lambda_s) = 0$
\end{defn}

\begin{prop}\label{prop:NGeodesicEquivalentDefn}
    Given a unit-speed curve $\gamma : [a,b] \rightarrow M$ in $\mathcal{P}(N)$, the following are equivalent.
    \begin{itemize}
        \item $\gamma$ is an $N$-geodesic
        \item $\gamma$ is a critical point of the restricted length functional $L|_{\mathcal{P}(N)}$ with respect to proper variations.
        \item $\gamma(t) = \exp^\nu(t \mathbf{v})$, where $p = \gamma(0)$ and $\mathbf{v} = \dot\gamma(a) \in S(\nu_p)$.
    \end{itemize}
\end{prop}

For proof of the above equivalences, we refer to \cite{Zhao2017}. Next, we define the notion of global distance minimizer from $N$.

\begin{defn}\label{defn:NSegment}
    An $N$-geodesic $\gamma : [0, \ell] \to M$ is called an \emph{$N$-segment} joining $N$ to $\gamma(\ell)$ if $d(N, \gamma(\ell)) = L(\gamma) = \ell$.
\end{defn}

We have the following useful lemma which can be compared to \cite[Theorem 5.16, pg. 24]{Bus55}. By the notation $\gamma_i \rightarrow \gamma$, for unit-speed geodesics defined on possibly different domains, we mean $\gamma_i(t) \rightarrow \gamma(t)$ for all $t \ge 0$, which is justifiable under our assumption of forward completeness. Alternatively, the convergence can be understood by reparametrizing all the curves to a fixed domain, we refer to \cite[pg. 24]{Bus55} for the details..

\begin{lemma}\label{lemma:convergentSubsequenceNSegment}
    Suppose, $\gamma_i : [0, \ell_i] \rightarrow M$ are unit-speed minimizers joining $p_i = \gamma_i(0)$ to $q_i = \gamma_i(\ell_i)$, where $\ell_i = d(p_i, q_i)$. Suppose $\ell_i \rightarrow \ell$. If either
    \begin{center}
        (a) $p_i \rightarrow p$ and $F$ is forward complete, \kern.5cm or \kern.5cm (b) $q_i \rightarrow q$ and $F$ is backward complete,
    \end{center}
    then a subsequence of $\gamma_i$ converges uniformly to a minimizer $\gamma$, which satisfies $L(\gamma) = \ell$.
\end{lemma}
\begin{proof}
    Suppose (a) holds. Consider the curves $\hat{\gamma}_i : [0, \ell] \rightarrow M$ by $\hat{\gamma}_i(t) = \gamma_i \left( t \frac{\ell_i}{\ell}  \right)$. Since $\gamma_i$ is a minimizer, it follows that the family $\left\{ \hat{\gamma}_i \right\}$ is uniformly (forward) Lipschitz. Indeed, for $0 \le t_1 \le t_2 \le \ell$ and for all $i$ we have,
    \[d\left( \hat{\gamma}_i(t_1), \hat{\gamma}_i(t_2) \right) = d \left( \gamma_i \left( t_1 \frac{\ell_i}{\ell}  \right), \gamma_i \left( t_2 \frac{\ell_i}{\ell}  \right) \right) = \left( t_2 - t_1 \right)\frac{\ell_i}{\ell} \le \kappa \left( t_2 - t_1 \right),\]
    where the limit $\kappa \coloneqq \frac{1}{\ell} \sup_i \ell_i$ is finite since $\ell_i \rightarrow \ell < \infty$. Next, consider the forward closed ball $K \coloneqq \overline{B_+(p, \ell + 1)} = \left\{ x \;\middle|\; d(p, x) \le \ell + 1 \right\}$ which is compact by the Hopf-Rinow theorem. As $p_i \rightarrow p$ and $\ell_i \rightarrow \ell$, we have $d(p, p_i) < \frac{1}{3}$ and $\ell_i < \ell + \frac{1}{3}$ for $i$ large. Now, for any $0 \le t \le \ell$, we have $d(p_i, \hat{\gamma}_i(t)) = d\left( \gamma_i(0), \gamma_i \left( t \frac{\ell_i}{\ell}  \right) \right) = t\frac{\ell_i}{\ell} \le \ell_i < \ell + \frac{1}{3}$ and hence,
    \[d(p, \hat{\gamma}_i(t)) \le d(p, p_i) + d(p_i, \hat{\gamma}_i(t)) < \frac{1}{3} + \ell + \frac{1}{3}  < \ell + \frac{2}{3} < \ell + 1.\]
    In other words, $\hat{\gamma}_i$ is contained in $K$, and thus the family $\left\{ \hat{\gamma}_i \right\}_i$ is uniformly bounded. But then, by an application of the Arzel\'a-Ascoli theorem \cite{CoZi07}, a subsequence converges uniformly to some curve $\gamma : [0, \ell] \rightarrow M$, joining $p = \lim p_i$ to $q \coloneqq \lim q_i$. As $\gamma$ is the uniform limit of Lipschitz functions, it is Lipschitz, and thus, almost everywhere differentiable. By \cite[(5.8), pg. 20]{Bus55}, $L(\gamma) \le \lim \inf L(\gamma_i) = \ell$. On the other hand, $d(p, q) = \lim d(p_i, q_i) = \lim \ell_i = \ell$. Thus, $\gamma$ must be a smooth minimizer joining $p$ to $q$, with $L(\gamma) = \ell = \lim L(\gamma_i)$.

    Now, suppose (b) holds. But then, the reversed curve $\bar{\gamma}_i : [0, \ell_i]\rightarrow M$ given by $\bar{\gamma}_i(t) = \gamma_i(\ell - t)$ is a geodesic with respect to the reverse metric $\bar{F}$, which is now forward complete by hypothesis. The proof follows as in (a).
\end{proof}

As an immediate consequence, we have the following.

\begin{prop}\label{prop:existenceOfNSegements}
    Suppose $(M, F)$ is a forward complete Finsler manifold, and $N$ is a closed submanifold of $M$. Let $\gamma_i : [0, \ell_i] \rightarrow M$ be $N$-segments joining $N$ to $q_i \coloneqq  \gamma_i(\ell_i)$. Suppose that either
    \begin{equation}
        \tag{$\mathsf{H}$} \label{eq:hypothesisH}
        \text{(a) $N$ is compact, \kern.5cm or \kern.5cm (b) $F$ is backward complete as well.}
    \end{equation}
    If $\ell_i \rightarrow \ell$ and $q_i \rightarrow q$, then there exists an $N$-segment $\gamma : [0, \ell]\rightarrow M$ joining $N$ to $q$, with a subsequence $\gamma_i \rightarrow \gamma$. In particular, for any $q \in M$ there exists an $N$-segment $\gamma$ joining $N$ to $q$, with $d(N, q) = L(\gamma)$.
\end{prop}
\begin{proof}
    Set $p_i = \gamma_i(0)$, and suppose $\gamma_i(t) = \exp^\nu(t \mathbf{v}_i)$ for $\mathbf{v}_i \in S(\nu_{p_i})$. If $N$ is compact, then we can get a subsequence $(p_i, \mathbf{v}_i) \rightarrow (p, \mathbf{v}) \in S(\nu)$, and we can appeal to \autoref{lemma:convergentSubsequenceNSegment}. Otherwise, if $F$ is backward complete, again \autoref{lemma:convergentSubsequenceNSegment} applies as $q_i \rightarrow q$ is given. In any case, we have a minimizer $\gamma : [0, \ell] \rightarrow M$ joining $p$ to $q$, with $\gamma_i \rightarrow \gamma$, where $p$ being an accumulation point of $p_i$, belongs to $N$. Note that, $d(N, q) = \lim d(N, q_i) = \lim \ell_i = \ell = L(\gamma)$. A standard argument using the first variation principle \cite{Zhao2017} shows that $\dot\gamma(0) \in S(\nu_p)$, and thus $\gamma$ is an $N$-segment joining $N$ to $q$.

    Now, suppose $q \in M \setminus N$, and let $\ell \coloneqq d(N, q) = \inf_{p \in N} d(p, q)$. Choose $p_i \in N$ so that $\ell_i \coloneqq d(p_i, q) \rightarrow \ell$. Let $\gamma_i : [0, \ell_i] \rightarrow M$ be a minimizer joining $p_i = \gamma_i(0)$ to $q = \gamma_i(\ell_i)$. As argued before, we then get an $N$-segment $\gamma : [0, \ell] \rightarrow M$ joining $N$ to $q$, with $d(N, q) = \ell = L(\gamma)$.
\end{proof}

\begin{remark}\label{rmk:distanceNotAchieved}
    As the notions of forward and backward completeness coincide in a complete Riemannian manifold, hypothesis (\hyperref[eq:hypothesisH]{H}) is automatically satisfied for a topologically closed submanifold. As such, one can remove the compactness assumption on the submanifold from most of the results in \cite{BaPr21}. On the other hand, for a general Finsler manifold, in the absence of the hypothesis (\hyperref[eq:hypothesisH]{H}), the distance from a closed submanifold to a point may not even be achieved on the submanifold. Indeed, consider the following example.
    \begin{figure}[H]
        \centering
    \def\svgwidth{0.4\columnwidth}
\begingroup%
  \makeatletter%
  \providecommand\color[2][]{%
    \errmessage{(Inkscape) Color is used for the text in Inkscape, but the package 'color.sty' is not loaded}%
    \renewcommand\color[2][]{}%
  }%
  \providecommand\transparent[1]{%
    \errmessage{(Inkscape) Transparency is used (non-zero) for the text in Inkscape, but the package 'transparent.sty' is not loaded}%
    \renewcommand\transparent[1]{}%
  }%
  \providecommand\rotatebox[2]{#2}%
  \newcommand*\fsize{\dimexpr\f@size pt\relax}%
  \newcommand*\lineheight[1]{\fontsize{\fsize}{#1\fsize}\selectfont}%
  \ifx\svgwidth\undefined%
    \setlength{\unitlength}{384.3899376bp}%
    \ifx\svgscale\undefined%
      \relax%
    \else%
      \setlength{\unitlength}{\unitlength * \real{\svgscale}}%
    \fi%
  \else%
    \setlength{\unitlength}{\svgwidth}%
  \fi%
  \global\let\svgwidth\undefined%
  \global\let\svgscale\undefined%
  \makeatother%
  \begin{picture}(1,0.88860519)%
    \lineheight{1}%
    \setlength\tabcolsep{0pt}%
    \put(0,0){\includegraphics[width=\unitlength,page=1]{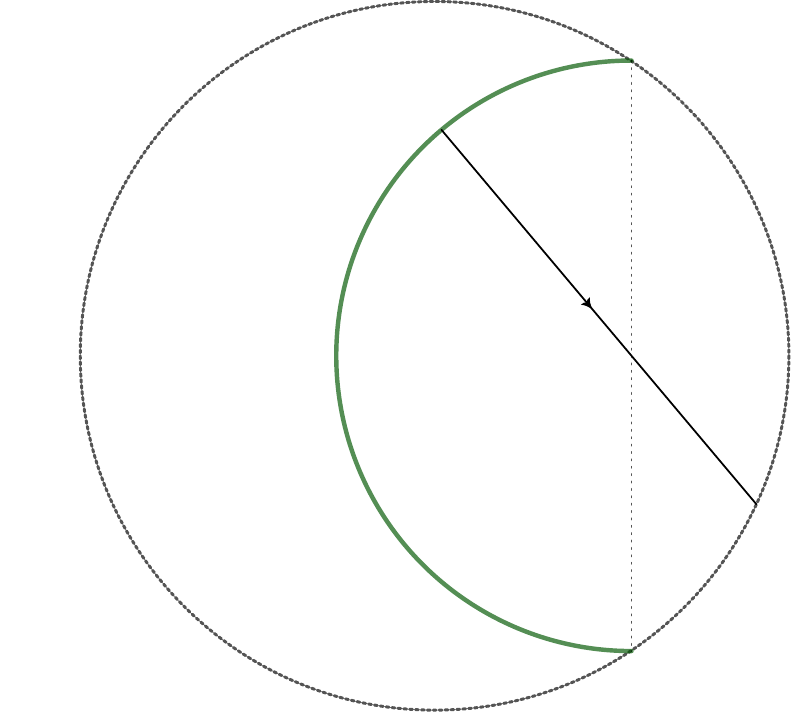}}%
    \put(0.80716188,0.837492){\color[rgb]{0,0,0}\makebox(0,0)[lt]{\lineheight{1.25}\smash{\begin{tabular}[t]{l}$a_1$\end{tabular}}}}%
    \put(0.79954398,0.02903956){\color[rgb]{0,0,0}\makebox(0,0)[lt]{\lineheight{1.25}\smash{\begin{tabular}[t]{l}$a_2$\end{tabular}}}}%
    \put(0.80583072,0.44381869){\color[rgb]{0,0,0}\makebox(0,0)[lt]{\lineheight{1.25}\smash{\begin{tabular}[t]{l}$x$\end{tabular}}}}%
    \put(0.52829612,0.75961502){\color[rgb]{0,0,0}\makebox(0,0)[lt]{\lineheight{1.25}\smash{\begin{tabular}[t]{l}$p$\end{tabular}}}}%
    \put(0.96251424,0.24891061){\color[rgb]{0,0,0}\makebox(0,0)[lt]{\lineheight{1.25}\smash{\begin{tabular}[t]{l}$a$\end{tabular}}}}%
    \put(0.31824026,0.40286294){\color[rgb]{0,0,0}\makebox(0,0)[lt]{\lineheight{1.25}\smash{\begin{tabular}[t]{l}$N$\end{tabular}}}}%
    \put(-0.00121946,0.61221983){\color[rgb]{0,0,0}\makebox(0,0)[lt]{\lineheight{1.25}\smash{\begin{tabular}[t]{l}$\partial \Omega$\end{tabular}}}}%
    \put(0,0){\includegraphics[width=\unitlength,page=2]{funk.pdf}}%
  \end{picture}%
\endgroup%

        \caption{Funk metric on a disc in $\mathbb{R}^2$}
        \label{fig:funk}
    \end{figure}
    In \autoref{fig:funk}, consider $\Omega$ to be an open disc of radius $R$ in $\mathbb{R}^2$ with the Funk metric $d_\Omega$ \cite[Section 6.5]{Ohta2021}. It is well-known that $d_\Omega$ is induced by a Finsler metric on $\Omega$, which is forward complete but \emph{not} backward complete. Let $a_1 a_2$ be a chord of length $2r$ with midpoint $x$, and $N$ be the semicircle with $a_1 a_2$ as diameter, without the endpoints $\left\{ a_1, a_2 \right\}$. Note that $N$ is closed in the subspace topology of $\Omega$, but not compact. Now, for some $p \in N$ let $a$ be the intersection of $\partial \Omega$ with the ray joining $p$ to $x$. Denoting $\left\lVert \cdot \right\rVert$ as the Euclidean distance, we have 
    \[d_\Omega(p, x) = \log \left( \frac{\left\lVert p - a \right\rVert}{\left\lVert x - a \right\rVert}  \right) = \log \left( \frac{\left\lVert p - x \right\rVert + \left\lVert x - a \right\rVert}{\left\lVert x - a \right\rVert} \right) = \log \left( 1 + \frac{r}{\left\lVert x - a \right\rVert}  \right),\]
    which attains the infimum $\log \left( 1 + \frac{r}{r}  \right) = \log 2$ as $p \rightarrow  a_1$ or $a_2$. Thus, $d(N, x) = \log 2$, but the distance is not achieved on $N$.
\end{remark}

Similar to \autoref{prop:globalMinimizerIsLocalMinimizer}, we have the following result.
\begin{prop}\label{prop:distanceMinimizingNormal}
    An $N$-geodesic $\gamma : [0,\ell] \rightarrow  M$ is an $N$-segment if and only if $d(N, \gamma(t)) = t$ for all $0 \le t \le \ell$.
\end{prop}
\begin{proof}
    Suppose $\gamma$ is an $N$-segment. Then, $\gamma$ is in particular a minimizer joining $\gamma(0)$ to $\gamma(\ell)$. If possible, say, there exists some $0 < T < \ell$ so that $d(N, \gamma(T)) < T = L(\gamma|_{[0,T]})$. Then, there exists some $N$-segment $\eta : [0,T_{0}] \rightarrow  M$ joining $N$ to $q = \gamma(T) = \eta(T_{0})$, so that $L(\eta) = d(N, \gamma(T)) < L(\gamma|_{[0,T]})$. Consider the \emph{broken} curve $\zeta$ obtained by concatenating $\eta$ and $\gamma|_{[T,\ell]}$. It follows that 
    \[L(\zeta) = L(\eta) + L(\gamma|_{[T,\ell]}) < L(\gamma|_{[0,T]}) + L(\gamma|_{[T,\ell]}) = L(\gamma) = d(N, \gamma(\ell)),\]
    which contradicts the fact that $\zeta$ is a curve joining $N$ to $\gamma(\ell)$. The converse statement is obvious.
\end{proof}

\subsubsection{\texorpdfstring{$N$}{N}-Jacobi Fields}
\begin{defn}\label{defn:NGeodesicVariation}
    Given a unit-speed $N$-geodesic $\gamma : [a,b]\rightarrow M$, a variation $\Lambda : (-\epsilon, \epsilon) \times [a,b] \rightarrow M$ is said to be an \emph{$N$-geodesic variation} if for each $-\epsilon < s < \epsilon$, the curve $\Lambda_s : [a,b] \rightarrow M$ is a constant speed $N$-geodesic.
\end{defn}

\begin{defn}\label{defn:NJacobiField}
    Let $\gamma : [a,b] \rightarrow M$ be a unit-speed $N$-geodesic. A vector field $J \in \Gamma\gamma^*TM$ is called an $N$-Jacobi field of $\gamma$ if there exists an $N$-geodesic variation $\Lambda:(-\epsilon,\epsilon)\times [a,b]\rightarrow M$ satisfying 
    \[\left.\frac{\partial}{\partial s}\right|_{s=0} \Lambda(s,t) = J(t), \quad t\in[a,b].\]
\end{defn}

\begin{prop}\label{prop:NJacobiEquation}\cite{Javaloyes2015}
    Given a unit-speed $N$-geodesic $\gamma:[a,b]\rightarrow M$, a vector field $J \in \Gamma\gamma^*TM$ is an $N$-Jacobi field if and only if it satisfies 
    \[D^{\dot\gamma}_\gamma D^{\dot\gamma}_\gamma J - R^{\gamma}(\dot\gamma, J) \dot\gamma = 0, \quad J(a)\in T_{\gamma(a)} N, \quad \pi  D^{\dot\gamma}_\gamma J(a) = \pi \nabla^{\dot\gamma}_J \dot\gamma(a),\]
    where $\pi$ is the first component projection of the canonical splitting $T_{\gamma(a)}M = T_{\gamma(a)}N \oplus \left( T_{\gamma(a)}N \right)^{\perp_{g_{\dot\gamma(a)}}}$.
\end{prop}

From the second variation formula of the restricted energy functional $E|_{\mathcal{P}_N}$, we get the index form.

\begin{defn}\label{defn:NIndexForm}
    Let $\gamma : [a,b] \rightarrow M$ be a unit-speed $N$-geodesic. Then, the \emph{index form} $I_\gamma$ is defined for $X, Y \in T_\gamma \mathcal{P}_N$ with $X(b) = 0 = Y(b)$ as 
    \[\mathcal{I}_\gamma(X,Y) \coloneqq \int_a^b \left[ g_{\dot\gamma}(D^{\dot\gamma}_\gamma X, D^{\dot\gamma}_\gamma Y) - g_{\dot\gamma}\left( R^{\gamma}(\dot\gamma, X) Y, \dot \gamma \right) \right] - g_{\dot\gamma(a)} \left( \pi^\perp \nabla^{\dot\gamma(a)}_X Y, \dot\gamma(a) \right),\]
    where $\pi^\perp$ is the second component projection of the canonical splitting $T_{\gamma(a)}M = T_{\gamma(a)}N \oplus \left( T_{\gamma(a)}N \right)^{\perp_{g_{\dot\gamma(a)}}}$.
\end{defn}

It follows that $I_\gamma$ is a symmetric $2$-form, and the kernel of the index form $I_\gamma$ consists of precisely the $N$-Jacobi fields along $\gamma$.

\subsection{Cut Locus of a Submanifold}
\begin{defn}\label{defn:cutLocusP}
    The \emph{cut locus of a point} $p \in M$ is the set consisting of all $q \in M$ such that there exists a minimizer from $p$ to $q$, any extension of which fails to be distance minimizing. We denote the cut locus of $p$ by $\mathrm{Cu}(p)$.
\end{defn}

Generalizing the notion of a cut locus of a point to an arbitrary submanifold (or, more generally, any subset), we have the following definition.

\begin{defn}\label{defn:cutLocusN}
    Given a submanifold $N \subset M$, the \emph{cut locus} of $N$ consists of points $q \in M$ such that there exists an $N$-segment joining $N$ to $q$, whose extension fails to be an $N$-segment.
\end{defn}

\begin{prop}\label{prop:cutLocusAllSegmentsFail}
    $q \in \mathrm{Cu}(N)$ if and only if for \emph{any} $N$-segment $\gamma$ joining $N$ to $q$, an extension of $\gamma$ fails to be an $N$-segment.
\end{prop}
\begin{proof}
    Suppose $q \in \mathrm{Cu}(N)$. We then have an $N$-segment $\gamma : [0,\ell] \rightarrow  M$ with $\ell = d(N, q)$ and $q = \gamma(\ell)$, whose extension fails to be an $N$-segment. If possible, let $\eta : [0, \ell] \rightarrow  M$ be another $N$-segment joining $N$ to $q = \eta(\ell)$, which remains an $N$-segment when extended to $\eta : [0, \ell + \epsilon] \rightarrow  M$ for some $\epsilon > 0$. Since $\gamma, \eta$ are distinct, we have $\dot\gamma(\ell) \ne \dot\eta(\ell)$. Thus, we have a \emph{broken} curve $\zeta$ obtained from concatenating $\gamma$ and $\eta|_{[\ell, \ell + \epsilon]}$. Now, $L(\zeta) = L(\gamma) + L(\eta|_{[\ell, \ell + \epsilon]}) = \ell + \epsilon = d(N, \eta(\ell + \epsilon))$, i.e., $\zeta$ achieves the minimal distance. But then $\zeta$ must be a smooth curve, which is a contradiction.
    
    The converse statement is obvious.
\end{proof}

\begin{defn}\label{defn:cutTime}
    Given $\mathbf{v} \in S(\nu)$, the \emph{cut time} of $\mathbf{v}$ is defined as 
    \[\rho(\mathbf{v}) \coloneqq \sup \left\{ t \;\middle|\; d(N, \gamma_{\mathbf{v}}(t \mathbf{v})) = t \right\},\]
    where $\gamma_{\mathbf{v}}(t) = \exp^\nu(t \mathbf{v})$ is the unique $N$-geodesic with initial velocity $\mathbf{v}$.
\end{defn}

Note that we allow $\rho$ to take the value $\infty$, although if $M$ is compact then $\rho(\mathbf{v}) < \infty$ for any $\mathbf{v} \in S(\nu)$ (\autoref{obs:compactnessAndTangentCutLocus} (\hyperref[obs:compactness:MCompactImpliesRhoFinite]{2})). Indeed, we shall see in \autoref{thm:rhoContinuous} that the map $\rho: S(\nu ) \rightarrow  [0, \infty]$ is continuous, and consequently, $\rho$ is finite if $N$ is compact (\autoref{obs:compactnessAndTangentCutLocus}(\hyperref[obs:compactness:NCompactRhoFiniteIMpliesMCompact]{3})). It follows from \cite{Alves2019} that $\rho$ is always strictly positive for any $N$ not necessarily compact.

\begin{defn}\label{defn:tangentCutLocus}
    Given a submanifold $N \subset M$, the \emph{tangent cut locus} of $N$ is defined as 
    \[\widetilde{\mathrm{Cu}}(N) \coloneqq \left\{ \rho(\mathbf{v}) \mathbf{v} \;\middle|\; \mathbf{v}\in S(\nu), \; \rho(\mathbf{v}) \ne \infty \right\} \subset \nu.\]
\end{defn}

It follows from \autoref{defn:cutLocusN} that $\mathrm{Cu}(N) = \exp^\nu(\widetilde{\mathrm{Cu}}(N))$. 

\subsection{Focal Locus of a Submanifold} 
\begin{defn}\label{defn:tangentFocalLocus}
    Given a submanifold $N \subset M$, a vector $\mathbf{v}\in \hat{\nu}$ is said to be a \emph{tangent focal point} of $N$ if 
    \[d(\exp^\nu)|_{\mathbf{v}} : T_{\mathbf{v}} \hat{\nu} \rightarrow T_{\exp^\nu(\mathbf{v})}M\]
    is degenerate, i.e., if $\mathbf{v}$ is a critical point of $\exp^\nu|_{\hat{\nu}}$. The nullity of the map is known as the \emph{index} of the tangent focal point $\mathbf{v}$. The collection of all such vectors is called the \emph{tangent focal locus} of $N$.
\end{defn}

\begin{defn}\label{defn:focalTime}
    Given $N \subset M$ and $\mathbf{v} \in S(\nu)$, we define the \emph{(first) focal time} as 
    \[\lambda(\mathbf{v}) = \inf \left\{ t \;\middle|\; \text{$d(\exp^\nu)|_{t \mathbf{v}}$ is degenerate} \right\}.\]
    We set $\lambda(\mathbf{v}) = \infty$ if there are no focal points of $N$ in the direction of $\mathbf{v}$.
\end{defn}

\begin{defn}\label{defn:focalLocus}
    The \emph{focal locus} of $N$ is defined as the image of the tangent focal locus under the $\exp^\nu$ map. In other words, the focal locus of $N$ consists of the critical values of the (restricted) normal exponential map $\exp^\nu|_{\hat{\nu}}$. The \emph{first focal locus} of $N$ is defined as the set $\left\{ \exp^\nu(\lambda(\mathbf{v}) \mathbf{v}) \mid \mathbf{v} \in S(\nu), \lambda(\mathbf{v}) \ne \infty \right\}$.
\end{defn}

In the literature, the focal locus of a submanifold (of a Riemannian manifold) is also known as the conjugate locus. In this article, we reserve the term conjugate locus for when $N$ is a point. The following result is a natural generalization of \autoref{prop:conjugateLocusJacobiField}, a proof of which can be deduced from \cite[Proposition 4.5]{Zhao2017}.

\begin{prop}\label{prop:focalLocusJacobiField} Let $\mathbf{v} \in S(\nu)$ and $\gamma$ be the unit-speed $N$-geodesic given as $\gamma(t) = \exp^\nu(t \mathbf{v})$. Then for any $t > 0$, the kernel of the map 
    \[d\left( \exp^\nu|_{\hat{\nu}} \right)\big|_{t\mathbf{v}} : T_{t\mathbf{v}} \hat{\nu} \rightarrow T_{\gamma(t)} M\]
    is naturally isomorphic to the vector space 
    \[\left\{ J \;\middle|\; \text{$J$ is an $N$-Jacobi field along $\gamma$, satisfying $J(t) = 0$} \right\}.\]
\end{prop}
\section{Main Results} \label{sec:mainResults}

We start with an important characterization of the cut locus in terms of the focal locus (\autoref{thm:cutPointClassification}). A result of Basu and Prasad \cite{BaPr21} is then generalized to prove that the cut locus is the closure of the separating set, and hence it is closed (\autoref{thm:separatingSetDense}). Next, we will see that the complement of the cut locus of a submanifold strongly deforms onto the submanifold, and dually, the complement of a submanifold deforms onto its cut locus (\autoref{thm:deformation}). Lastly, we prove the existence of $N$-geodesic loops (\autoref{thm:NGeodesicLoop}) which is a generalization of the results of \cite{Innami2012,Xu2015}. Note that in the light of \autoref{rmk:distanceNotAchieved}, henceforth for most of the results we require condition (\hyperref[eq:hypothesisH]{H}) to hold true.

\subsection{Characterization of Cut Locus}
In this section, we discuss two ways one can characterize the cut locus of a submanifold, which are well-known in the Riemannian setting. Recall first the definition.
\begin{defn}\label{defn:separtingSet}
    Given a submanifold $N \subset M$, a point $p \in M$ is said to be a \emph{separating point} \footnote{The notation $\mathrm{Se}(N)$ was introduced by F.E. Wolter in \cite{Wol79}. In a personal communication, Wolter has commented that ``I think that I was the first to introduce the notation $\mathrm{Se}(\cdot)$ ... simply as derived from `\emph{Several} geodesics'.''} of $N$ if there exist two distinct $N$-segments joining $N$ to $p$. The collection of all separating points of $N$ is called the \emph{separating set} (or the \emph{Maxwell strata}) of $N$, denoted $\mathrm{Se}(N)$.
\end{defn}

\begin{prop}\label{prop:separtingSetContainedInCutLocus}
    Let $N$ be a submanifold of $(M, F)$. Then, $\mathrm{Se}(N) \subset \mathrm{Cu}(N)$.
\end{prop}
\begin{proof}
    Let $q \in \mathrm{Se}(N)$, and take two $N$-segments $\gamma_1, \gamma_2 : [0, \ell] \to M$ joining, respectively, $p_1, p_2 \in N$ to $q$. If possible, suppose $q \not\in \mathrm{Cu}(N)$. Then, a sufficiently small extension of $\gamma_2$, say, $\gamma_2 : [0, \ell + \epsilon] \to M$ remains an $N$-segment (\autoref{prop:cutLocusAllSegmentsFail}), and hence, $d(N, \gamma_2(\ell + \epsilon)) = \ell + \epsilon$.
    \begin{figure}[H]
        \centering
    \def\svgwidth{0.5\columnwidth}
    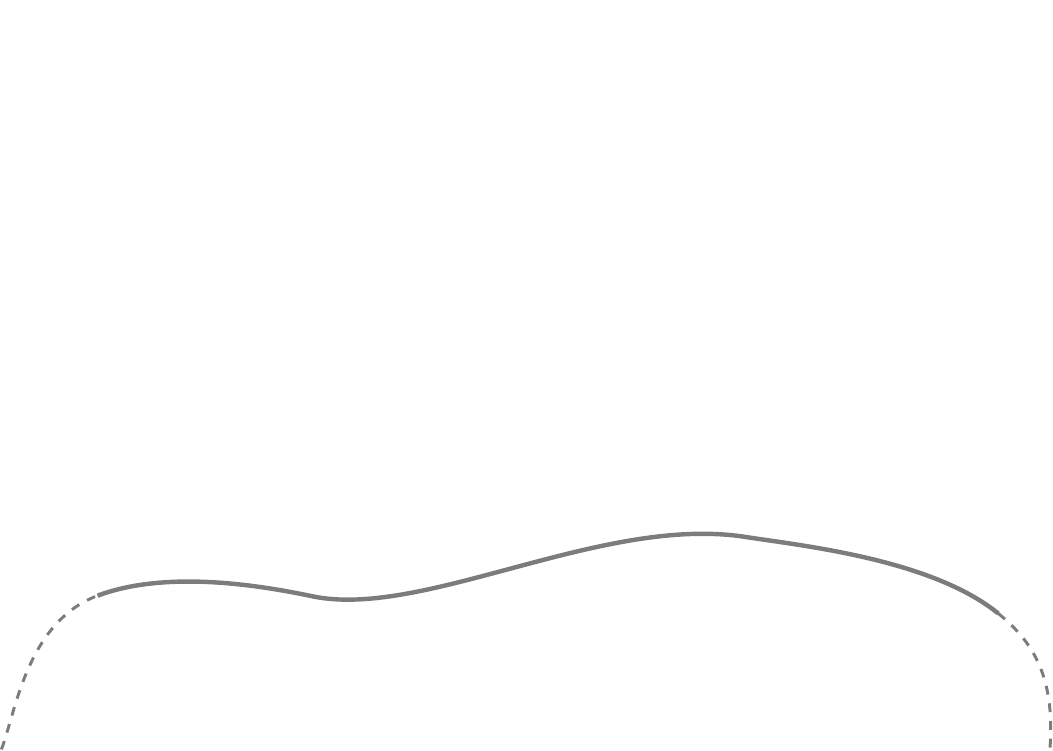

        \caption{Constructing the broken curve $\eta$.}
        \label{fig:Se_is_contained_in_Cu}
    \end{figure}
    On the other hand $\dot\gamma_1(\ell) \ne \dot\gamma_2(\ell)$ as they are distinct geodesics, and so we have a broken curve $\eta : [0, \ell + \epsilon] \to M$ obtained from concatenating $\gamma_1$ and $\gamma_2|_{[\ell, \ell + \epsilon]}$ (\autoref{fig:Se_is_contained_in_Cu}). In particular, from the triangle inequality, we get 
    \[d(p_1, \eta(\ell + \epsilon)) < d(p_1, \eta(\ell)) + d(\eta(\ell), \eta(\ell + \epsilon)) = \ell + \epsilon.\]
    Consequently, \[d(N, \gamma_2(\ell + \epsilon)) = d(N, \eta(\ell + \epsilon)) \le d(p_1, \eta(\ell + \epsilon)) < \ell + \epsilon.\]
    This is a contradiction and so, $q \in \mathrm{Cu}(N)$.
\end{proof}

In fact, for any point on an $N$-segment not in the cut locus, we have the following uniqueness result.
\begin{prop}\label{prop:uniquenessOfGeodesicBeforeCutPoint}
    Let $N$ be a submanifold of a Finsler manifold $(M, F)$. Let $q \in \mathrm{Cu}(N)$ and $\gamma : [0, \ell] \to M$ be an $N$-segment joining $p \in N$ to $q$. Then, for any $0 < t_0 < \ell$, $\gamma|_{[0, t_0]}$ is the unique $N$-segment joining $N$ to $\gamma(t_0)$.
\end{prop}
\begin{proof}
    Suppose, if possible, there exists some $N$-segment $\eta : [0, t_0] \to M$ joining $N$ to $q_0 = \gamma(t_0)$, which is distinct from $\gamma|_{[0, t_0]}$. Then $q_0 \in \mathrm{Se}(N) \subset \mathrm{Cu}(N)$ by \autoref{prop:separtingSetContainedInCutLocus}. On the other hand, $\gamma$ remains an $N$-segment beyond $\gamma(t_0)$, and so $q_0 \not\in \mathrm{Cu}(N)$ by \autoref{prop:cutLocusAllSegmentsFail}. This is a contradiction.
\end{proof}

Let us now recall the following useful result, which is known for the case when $N$ is a point \cite[Proposition 7.16, pg. 89]{Ohta2021}.
\begin{lemma}\label{lemma:beyondFocalPoint}
    Let $\gamma : [0, \ell] \to M$ be a unit speed $N$-geodesic. If $\gamma(T)$ is a focal point of $N$ along $\gamma$ for some $T \in (0, \ell)$, then $d(N, \gamma(t)) < t$ for all $t \in (T, \ell]$. In particular, if $\gamma|_{[0,T]}$ is an $N$-segment, then $\gamma(T) \in \mathrm{Cu}(N)$.
\end{lemma}
\begin{proof}
    Without loss of generality, we may assume that $\gamma(T)$ is the first focal point of $N$ along $\gamma$, and we show that $\gamma$ fails to be distance minimizing.
    
    First, let us choose $\epsilon > 0$ so that $\gamma|_{[T- \epsilon, T+\epsilon]}$ has no pair of conjugate points. Indeed, we can choose a strictly geodesically convex neighborhood, say, $U$ around $\gamma(T)$ so that any two points in $U$ can be joined by a unique minimizer which is furthermore completely contained in $U$. Let us choose $\epsilon > 0$ so that $\eta \coloneqq \gamma|_{[T-\epsilon, T+\epsilon]}$ is contained in $U$. Since by our assumption $\gamma$ is an $N$-segment beyond $T$, this segment $\eta$ is a minimizer between any pair of points on $\eta$, which is unique by the convexity of $U$. Now, if there is a pair of conjugate points on $\eta$, then $\eta$ fails to be distance minimizing between them \cite[Prop. 7.4.1]{Bao2000}, which is a contradiction. Thus, there are no pairs of conjugate points along $\gamma|_{[T-\epsilon, T+ \epsilon]}$.
    
    Now, by \autoref{prop:focalLocusJacobiField}, let $J_{1}$ be a non-zero $N$-Jacobi field along $\gamma|_{[0, T]}$ satisfying $J_{1}(T) = 0$. Note that the index form $I(J_1, J_1) = 0$. We extend $J_1$ by $0$ to all of $\gamma$, and clearly $I(J_1, J_1) = I(J_1|_{[0, T]}, J_1|_{[0, T]}) = 0$. Since there are no pairs of conjugate points, choose the (unique) Jacobi field $J_{2}$ along $\gamma|_{[T-\epsilon, T+\epsilon]}$ satisfying $J_{2}(T-\epsilon) = J_{1}(T-\epsilon)$ and $J_{2}(T+\epsilon)= 0 = J_{1}(T + \epsilon)$. Clearly $J_{2} \ne 0$ on $[T-\epsilon, T+\epsilon)$, as $J_2(s) = 0$ for some $T-\epsilon \le s < T + \epsilon$ will imply $\gamma(s)$ is conjugate to $\gamma(T+\epsilon)$ along $\gamma$, a contradiction. Next, consider the piecewise smooth vector field $J$ along $\gamma$ defined as follows 
    \[J(t) = \begin{cases}
        J_{1}(t), &t \in [0, T-\epsilon] \\
        J_{2}(t), &t \in [T-\epsilon, T + \epsilon] \\
        0, & t \in [T + \epsilon, \ell].
    \end{cases}\]
    Note that $J$ is different from $J_1$ only on $(T-\epsilon, T+ \epsilon)$. As $J|_{[T-\epsilon, T+\epsilon]} = J_2$ is a non-zero Jacobi field, we have from \cite[Lemma 7.14, pg. 88]{Ohta2021}
    \[I(J, J) - I(J_1, J_1) = I\left( J_2, J_2 \right) - I\left( J_1|_{[T-\epsilon, T+\epsilon]}, J_1|_{[T-\epsilon, T+\epsilon]} \right) < 0 \Rightarrow I(J, J) < I(J_1, J_1) = 0.\]
    Now, suppose $\gamma_{s} : [0, T] \rightarrow M$ is a $J$-variation of $\gamma$, so that $\gamma_0 = \gamma$. Since $\gamma$ is a geodesic, we have $E^{\prime}(0) = \frac{d}{d s}|_{s=0} E(\gamma_s) = 0$, and then it follows that $E^{\prime\prime}(0) = \frac{d^2}{d s^2}|_{s = 0} E(\gamma_s) = I(J,J) < I(J_{1}, J_{1}) = 0$. Hence, for any $s \ne 0$ sufficiently near 0, we have $E(\gamma_s) < E(\gamma_0)$. As $\gamma$ is a unit-speed geodesic, we have 
    \[L(\gamma) = L(\gamma_0) = \sqrt{\ell E(\gamma_0)} > \sqrt{\ell E(\gamma_s)} \ge L(\gamma_s).\]
    Thus, $\gamma$ fails to be distance minimizing beyond $T$, which concludes the proof.
\end{proof}

Recall that we have $\rho$ as the cut time (\autoref{defn:cutTime}), and $\lambda$ as the focal time (\autoref{defn:focalTime}) map. Then, as an immediate corollary, we have the following. 
\begin{corollary}\label{cor:cutTimeLessThanFocalTime}
    $\rho(\mathbf{v}) \le \lambda(\mathbf{v})$ for all $\mathbf{v}\in S(\nu)$.
\end{corollary}

The next result is well-known in the Riemannian setting \cite{BaPr21}\footnote{During the review process of this article, it came to the authors' attention that \autoref{thm:cutPointClassification} and \autoref{thm:rhoContinuous} are also proved in \cite[Prop. 5.1, 5.2]{Wu14} under the compactness hypothesis on $N$ (also see \autoref{rmk:distanceNotAchieved}).}.
\begin{thm}\label{thm:cutPointClassification}
    Let $N$ be a closed submanifold of a forward complete Finsler manifold $(M, F)$. Additionally assume (\hyperref[eq:hypothesisH]{H}) holds. Then, $q \in \mathrm{Cu}(N)$ if and only if $\gamma : [0, \ell] \to M$ is an $N$-segment with $q = \gamma(\ell)$ and at least one of the following holds true.
    \begin{enumerate}
        \item There exist two distinct $N$-segments joining $N$ to $q$, i.e, $q \in \mathrm{Se}(N)$, or
        \item $q$ is the first focal point of $N$ along $\gamma$.
    \end{enumerate}
\end{thm}
\begin{proof}
    Suppose $q \in \mathrm{Se}(N)$. Then by \autoref{prop:separtingSetContainedInCutLocus}, we have $q \in \mathrm{Cu}(N)$. On the other hand, suppose $q$ is the first focal point of $N$ along $\gamma$, and then by \autoref{lemma:beyondFocalPoint}, any extension of $\gamma$ fails to be distance minimizing from $N$. As $\gamma$ is an $N$-segment, we have $q \in \mathrm{Cu}(N)$.

    \begin{figure}[H]
        \centering
    \def\svgwidth{0.5\columnwidth}
    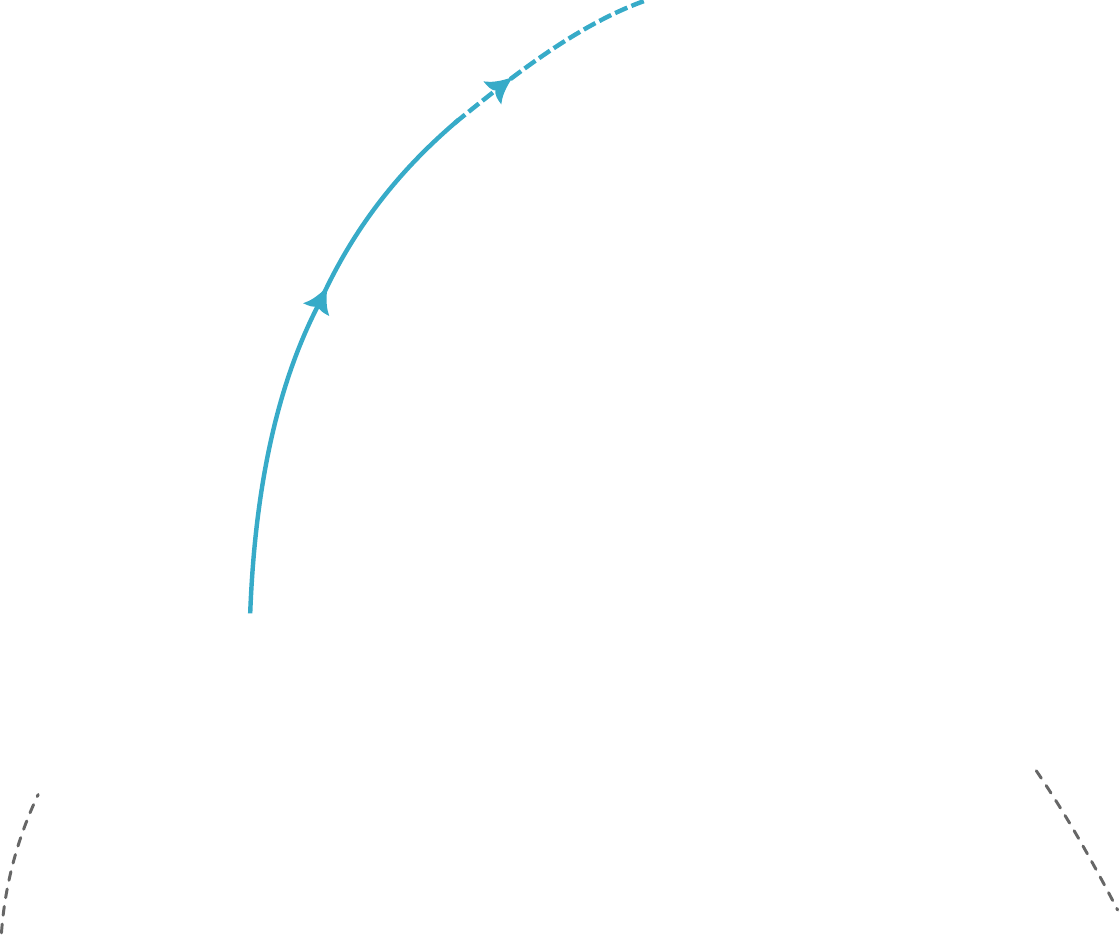

        \caption{Characterization of $\mathrm{Cu}(N)$}
        \label{fig:characterization_of_Cu-1}
    \end{figure}

    In order to prove the converse, let $q \in \mathrm{Cu}(N) \setminus \mathrm{Se}(N)$. Then, there exists a unique unit-speed $N$-segment $\gamma : [0,\ell] \to M$ joining $p = \gamma(0) \in N$ to $q = \gamma(\ell)$, where $\ell = d(N, q)$. Furthermore, any extension of $\gamma$ fails to be distance minimizing from $N$. Say, $\mathbf{v} = \dot \gamma(0) \in S(\nu)$, and then $\gamma(t) = \exp^\nu(t \mathbf{v})$. Let $q_i = \gamma\left( \ell + \frac{1}{i} \right)$. Then, $q_i \to q$ and $d(N, q_i) < L\left( \gamma|_{\left[ 0, \ell + \frac{1}{i} \right]} \right)$. Let $\gamma_i : [0, \ell_i] \to M$ be unit-speed $N$-segments joining $p_i = \gamma_i(0) \in N$ to $q_i = \gamma_i(\ell_i)$, where $\ell_i = d(N, q_i)$, and $\mathbf{v}_i = \dot\gamma_i(0) \in S(\nu_{p_i})$ (see \autoref{fig:characterization_of_Cu-1}). We have $\ell_i = d(N, q_i) \rightarrow d(N, q) = \ell$. Then, by \autoref{prop:existenceOfNSegements}, we have a subsequence $\gamma_i \rightarrow \gamma^\prime$, where $\gamma^\prime : [0, \ell] \rightarrow M$ is an $N$-segment joining $N$ to $q$. But then $\gamma = \gamma^\prime$, as $q \not \in \mathrm{Se}(N)$. In particular, $(p_i, \mathbf{v}_i) \rightarrow (p, \mathbf{v}) \Rightarrow \ell_i \mathbf{v}_i \rightarrow \ell \mathbf{v}$.  As $\gamma|_{\left[ 0, \ell + \frac{1}{i} \right]}$ is not an $N$-segment, $\gamma_i$ is distinct from $\gamma|_{\left[ 0, \ell + \frac{1}{i} \right]}$, and in particular, $\mathbf{v}_i \ne \mathbf{v} \Rightarrow \ell_i \mathbf{v}_i \ne \left( \ell + \frac{1}{i} \right)\mathbf{v}$. On the other hand,
    \[\exp^\nu(\ell_i \mathbf{v}_i) = \gamma_i(\ell_i) = q_i = \exp^\nu \left( \left(\ell + \frac{1}{i}\right)\mathbf{v} \right).\]
    Thus, the normal exponential map fails to be injective in any neighborhood of $\ell \mathbf{v}$. As a consequence of the inverse function theorem, we then have $d(\exp^\nu)|_{\ell \mathbf{v}}$ is degenerate, i.e., $q = \exp^\nu(\ell \mathbf{v})$ is a focal point of $N$ along $\gamma$. As $\gamma$ is an $N$-segment joining $N$ to $q$, it follows from \autoref{lemma:beyondFocalPoint} that $\gamma(t)$ is not a focal point of $N$ along $\gamma$ for $0 \le t < T$. Thus, $q$ is the first focal point of $N$ along $\gamma$. This completes the proof.
\end{proof}

\subsection{Continuity of the \texorpdfstring{$\rho$}{rho} map}
Recall the cut time map $\rho$ from \autoref{defn:cutTime}. We have the following result.

\begin{thm}\label{thm:rhoContinuous}
    Let $N$ be a closed submanifold of a forward complete Finsler manifold $(M, F)$. Suppose (\hyperref[eq:hypothesisH]{H}) holds. Then, the map $\rho : S(\nu) \to [0,\infty]$ is continuous.
\end{thm}
\begin{proof}
    Consider a converging sequence $(p_i, \mathbf{v}_i) \to (p, \mathbf{v})$ in $S(\nu)$. We show that $\rho(p_i,\mathbf{v}_i) \to \rho(p,\mathbf{v})$. Let $\ell_i = \rho(p_i, \mathbf{v}_i)$, and consider the $N$-segments $\gamma_i : [0, \ell_i] \to M$ starting at $p_i = \gamma_i(0)$, given by $\gamma_i(t) = \exp^\nu(t \mathbf{v}_i)$. If $\ell_i \ne \infty$, then we set $q_i = \gamma_i(\ell_i)$. From the definition of $\rho$, we have $\gamma_i$ is an $N$-segment for each $i$. Since $\rho(p_i, \mathbf{v}_i) \in [0,\infty]$, we must have some accumulation point, say, $ L_0  \in [0, \infty]$ for the set $\{ \rho(p_i, \mathbf{v}_i)\}$.
    
    Suppose, $L_0 = \infty$. Consider the $N$-geodesic $\gamma : [0, \infty) \to M$ defined as $\gamma(t) = \exp^\nu(t\mathbf{v})$. Passing to a subsequence, we have $\ell_i = \rho(p_i,\mathbf{v}_i) \to \infty$, and by assumption $\mathbf{v}_i = \dot \gamma_i(0) \to \mathbf{v} = \dot \gamma(0)$. But then, for any $c \ge 0$, we have $\ell_i = \rho(p_i, \mathbf{v}_i) \ge c$ for $i$ large, and hence
    \[d(N, \gamma(c)) = d(N, \exp^\nu(c \mathbf{v})) = \lim d(N, \exp^\nu(c \mathbf{v}_i)) = \lim d(N, \gamma_i(c)) = c.\]
    Thus, $\gamma$ is an $N$-segment for all time. By definition of $\rho$, we have $\rho(p, \mathbf{v}) = \infty = L_0$, and we are done.
        
    Suppose $ L_0  < \infty$, and so for infinitely many $(p_i,\mathbf{v}_i)$, we have $\ell_i < \infty$ and $q_i = \exp^\nu(\ell_i \mathbf{v}_i)$ is a cut point of $N$. It follows from \autoref{thm:cutPointClassification} that there are infinitely many $(p_i,\mathbf{v}_i)$ which satisfy one of the following.
    \begin{enumerate}[\quad(a)]
        \item \label{thm:rhoContinuous:case:firstFocalPoint} $q_i = \gamma_i(\ell_i) = \exp^\nu(\ell_i \mathbf{v}_i)$ is the first focal point of $N$ along $\gamma_i$.
        \item \label{thm:rhoContinuous:case:twoSegments} There exists some $N$-segment ending at $q_i$ distinct from $\gamma_i$.
    \end{enumerate}

    If (\hyperref[thm:rhoContinuous:case:firstFocalPoint]{a}) holds for infinitely many $(p_i,\mathbf{v}_i)$, then $\ker \left( d(\exp^\nu)|_{\ell_i \mathbf{v}_i} \right) \ne 0$. Since $\lim \ell_i =  L_0 $ is finite, passing to a subsequence, we may assume that there exists some $\epsilon > 0$ so that $L_0 - \epsilon < \ell_i <  L_0 + \epsilon$ for all $i$. Furthermore, we also assume that $p_i \in \overline{B_+(p, \epsilon)}$, which is compact by Hopf-Rinow theorem. Consider the subset \[\widetilde{\nu} \coloneqq \{ (a,\mathbf{v}) \in \nu \;|\; q\in \overline{B_+(p, \epsilon)}, \; L_0 -\epsilon \le F_p(\mathbf{v}) \le  L_0  + \epsilon \}.\]
    Clearly, $\widetilde{\nu}$ is a compact subset of $\hat{\nu}$, and furthermore, $\ell_i \mathbf{v}_i \in \widetilde{\nu}$. Let us now fix some auxiliary Riemannian metric, say, $g$ on $\hat{\nu}$, and choose unit vectors $\mathbf{w}_i \in \ker \left( d (\exp^\nu)|_{\ell_i \mathbf{v}_i} \right)$. Now, $\mathbf{w}_i$ is in the sub-bundle 
    \[E \coloneqq \bigcup_{(q,\mathbf{v}) \in \widetilde{\nu}} S^g\left( T_{(q,\mathbf{v})} \hat{\nu} \right) \subset S^g(T\hat{\nu}),\]
    where $S^g(.)$ denotes the unit sphere bundle with respect to $g$. As $E$ is compact, passing to a further subsequence, we may assume $\mathbf{w}_i \to \mathbf{w} \in E$. Then $0 = \lim d(\exp^\nu)|_{\ell_i \mathbf{v}_i} ( \mathbf{w}_i ) = d(\exp^\nu)|_{ L_0  \mathbf{v}}(\mathbf{w})$. As $\mathbf{w} \ne 0$, $q = \gamma(T) = \exp^\nu( L_0  \mathbf{v})$ is a focal point of $N$ along $\gamma$. Hence, any extension of $\gamma$ beyond $ L_0 $ fails to be distance minimizing (\autoref{lemma:beyondFocalPoint}), and we get $\rho(p,\mathbf{v}) =  L_0 $.

    If (\hyperref[thm:rhoContinuous:case:twoSegments]{b}) holds for infinitely many $(p_i, \mathbf{v}_i)$, there exists $(p_i^\prime, \mathbf{v}_i^\prime) \in S(\nu)$ so that the geodesic $\eta_i : [0,\ell_i] \to M$ defined by $\eta_i(t) = \exp^\nu(t \mathbf{v}_i^\prime)$ is distinct from $\gamma_i$, and is an $N$-segment satisfying $q_i = \eta_i(\ell_i)$. Since $q_i = \exp^\nu(\ell_i \mathbf{v}_i) \rightarrow \exp^\nu(L_0 \mathbf{v}) = q$, by \autoref{prop:existenceOfNSegements}, we have a subsequence $\eta_i \rightarrow \eta$, where $\eta : [0, L_0] \rightarrow M$ is an $N$-segment, given by $\eta(t) = \exp^\nu(t \mathbf{v}^\prime)$, for $\mathbf{v}^\prime = \dot \eta(0) \in S(\nu_{p^\prime})$. In particular, $(p_i, \mathbf{v}_i) \rightarrow (p^\prime, \mathbf{v}^\prime)$. If $\mathbf{v}^\prime \ne \mathbf{v}$, then we have $q \in \mathrm{Se}(N) \subset \mathrm{Cu}(N)$ (\autoref{prop:separtingSetContainedInCutLocus}), and we get $\rho(p,\mathbf{v}) =  L_0 $. Next, suppose $\mathbf{v}^\prime = \mathbf{v}$ and so, $\eta = \gamma$. But then $\exp^\nu$ fails to be injective in any neighborhood of $ L_0 \mathbf{v}$. Indeed, we have $(p_i^\prime, \mathbf{v}_i^\prime) \ne (p_i, \mathbf{v}_i) \Rightarrow \ell_i \mathbf{v}_i^\prime \ne \ell_i \mathbf{v}_i$, both of which are getting mapped to $q_i$ under the exponential map, and converging to $L_0 \mathbf{v}$. In particular, as a consequence of the inverse function theorem, $d(\exp^\nu)|_{ L_0  \mathbf{v}}$ is non-degenerate. Thus, $q = \exp^\nu( L_0  \mathbf{v})$ is a focal point of $N$ along $\gamma$. As $\gamma$ is an $N$-segment, by \autoref{lemma:beyondFocalPoint}, we get $\rho(p,\mathbf{v}) = L_0$.

    Thus, for any accumulation point $ L_0 $ of $\{\rho(p_i,\mathbf{v}_i)\}$ we have $ L_0  = \rho(p,\mathbf{v})$. In particular, $\rho(p_i,\mathbf{v}_i) \to \rho(p,\mathbf{v})$, which completes the proof.
\end{proof}

We can now prove the following.

\begin{thm}\label{thm:separatingSetDense}
    Let $N$ be a closed submanifold of a Finsler manifold $(M, F)$. Furthermore, assume (\hyperref[eq:hypothesisH]{H}) holds. Then, $\mathrm{Cu}(N) = \overline{\mathrm{Se}(N)}$.
\end{thm}
\begin{proof}
    Let $q \in \mathrm{Cu}(N) \setminus \mathrm{Se}(N)$. Then, there exists a unique $N$-segment $\gamma : [0, \ell] \to M$ joining $p = \gamma(0) \in N$ to $q= \gamma(\ell)$, and furthermore by \autoref{thm:cutPointClassification}, $q$ is the first focal point of $N$ along $\gamma$. We have $q = \exp^\nu(\ell \mathbf{v})$, where $\mathbf{v} = \dot \gamma(0) \in \nu_p$, and $\ell = \rho(\mathbf{v}) \ne \infty$.

    Fix some auxiliary Riemannian metric on $\hat{\nu} = \nu \setminus \{ \mathbf{0} \} $, which is a smooth manifold. Then, for some $\delta > 0$ sufficiently small, consider an open metric ball $V(\mathbf{v},\delta)$ of radius $\delta$ around $\mathbf{v}$ in $\nu$ (with respect to the auxiliary metric), with the closure denoted as $\overline{V(\mathbf{v},\delta)}$, so that
    \begin{enumerate}[\quad(i)]
        \item $\overline{V(\mathbf{v}, \delta)} \cap S(\nu)$ is homeomorphic to a closed Euclidean $(n-1)$-ball (\autoref{rmk:unitConeIsCompact}), where $n = \dim M$, and
        \item $\rho(\mathbf{w}) \ne \infty$ for all $\mathbf{w} \in \overline{V(\mathbf{v}, \delta)} \cap S(\nu)$, which can be justified from the continuity of $\rho$ (\autoref{thm:rhoContinuous}).
    \end{enumerate} 
    Define the set 
    \[\overline{U(\mathbf{v}, \delta)} = \left\{ \rho(\mathbf{w}) \mathbf{w} \;\middle|\; \mathbf{w} \in \overline{V(\mathbf{v},\delta)} \cap S(\nu) \right\}.\]
    As $\rho$ is finite and continuous, we have a surjective continuous map $\Phi: \overline{V(\mathbf{v}, \delta)} \cap S(\nu) \rightarrow \overline{U(\mathbf{v}, \delta)}$ defined as $\Phi(\mathbf{w}) = \rho(\mathbf{w})\mathbf{w}$ (see \autoref{fig:SeIsDense}).  Clearly, $\Phi$ is injective as well. Hence, by the invariance of domain, we have $\overline{U(\mathbf{v}, \delta)}$ is a closed Euclidean $(n-1)$-ball. Next, consider the set 
    \[\overline{\mathcal{U}(\mathbf{v}, \delta)} = \left\{ t \mathbf{x} \;\middle|\; \mathbf{x} \in \overline{U(\mathbf{v}, \delta)}, \; 0 \le t \le 1 \right\}.\]
    Then, $\overline{\mathcal{U}(\mathbf{v}, \delta)}$ is homeomorphic to the product of a closed Euclidean $(n-1)$-ball with the interval $[0, 1]$, which in turn is homeomorphic to a closed Euclidean $n$-ball. Taking a decreasing sequence $\delta_k \rightarrow 0$, we thus have a decreasing sequence of closed $n$-balls $\overline{\mathcal{U}}_k \coloneqq \overline{\mathcal{U}(\mathbf{v}, \delta_k)}$. Note that the point $\mathbf{u} = \ell \mathbf{v} = \Phi(\mathbf{v})$ lies on the boundary of $\overline{\mathcal{U}}_k$, and furthermore $\cap_k \overline{V(\mathbf{v}, \delta_k)} = \left\{ \mathbf{v} \right\}$. We now have a dichotomy.
    \begin{enumerate}[\quad(a)]
        \item \label{thm:separatingSetDense:case:injective} For some $\delta > 0$ small, $\exp^\nu|_{\overline{\mathcal{U}(\mathbf{v}, \delta)}}$ is injective, and consequently, again by an appeal to the invariance of domain, $W \coloneqq \exp^\nu\left( \overline{\mathcal{U}}_{k_0} \right)$ is a closed Euclidean $n$-ball in $M$ for some $k_0$ large.
        
        \item \label{thm:separatingSetDense:case:neverInjective} For any $\delta > 0$, $\exp^\nu|_{\overline{\mathcal{U}(\mathbf{v}, \delta)}}$ is not injective.
    \end{enumerate}

    If (\hyperref[thm:separatingSetDense:case:injective]{a}) appears, then we shall show that it will lead to a contradiction. Consider the open forward ball $B_{+}(q, \epsilon)$ around the point $q \in M$, which is homeomorphic to an open $n$-ball for $\epsilon > 0$ sufficiently small. Since $W$ is a closed Euclidean $n$-ball with $q = \exp^\nu(\mathbf{u})$ on the boundary, for $\epsilon$ sufficiently small we must have $B_+(q,\epsilon)$ intersects the complement $W$. Also, by \cite{Alves2019}, for $\epsilon$ sufficiently small $B_+(q, \epsilon) \cap N = \emptyset$ as well. Taking $\epsilon_i = \frac{1}{i}$ for $i$ large, we now pick points $q_i \in B_+(q, \epsilon_i) \setminus (W \cup N)$. Observe that $q_i \rightarrow q$. Let us consider $N$-segments $\gamma_i : [0, \ell_i] \rightarrow  M$ joining $p_i = \gamma_i(0)$ to $q_i = \gamma _i(\ell_i)$, with initial velocity $\mathbf{v}_i = \dot\gamma_i(0) \in S(\nu)$ and $\ell_i = d(N, q_i)$. Clearly, we have 
    \[\ell = d(N, q) = \lim d(N, q_i) = \lim \ell_i.\]
    By \autoref{prop:existenceOfNSegements}, a subsequence $\gamma_i \rightarrow \eta$, where $\eta : [0,\ell] \rightarrow M$ is an $N$-segment joining $N$ to $q$, given by $\eta(t) = \exp^\nu(t \mathbf{v}^\prime)$, where $\mathbf{v}^\prime = \dot\gamma(0) \in S(\nu_{p^\prime})$. As $q \in \mathrm{Cu}(N) \setminus \mathrm{Se}(N)$, we must have $\gamma$ and $\eta$ are the same. In particular, $(p, \mathbf{v}) = (p^\prime, \mathbf{v}^\prime)$. As $\mathbf{v}_i \rightarrow \mathbf{v}$, for some $i$ sufficiently large, we must have $\mathbf{v}_i \in V_{k_0} = V\left( \mathbf{v}, \frac{1}{k_0} \right)$. But then , $\ell_i \mathbf{v}_i \in \overline{\mathcal{U}}_{k_0}$, since $\ell_i = d(N, q_i) \le \rho(\mathbf{v}_i)$. This implies, $q_i = \exp^{\nu}(\ell_i \mathbf{v}_i) \in W$, a contradiction. Therefore, (\hyperref[thm:separatingSetDense:case:injective]{a}) cannot appear.

    \begin{figure}[!htbp]
        \centering
    \def\svgwidth{1\columnwidth}
    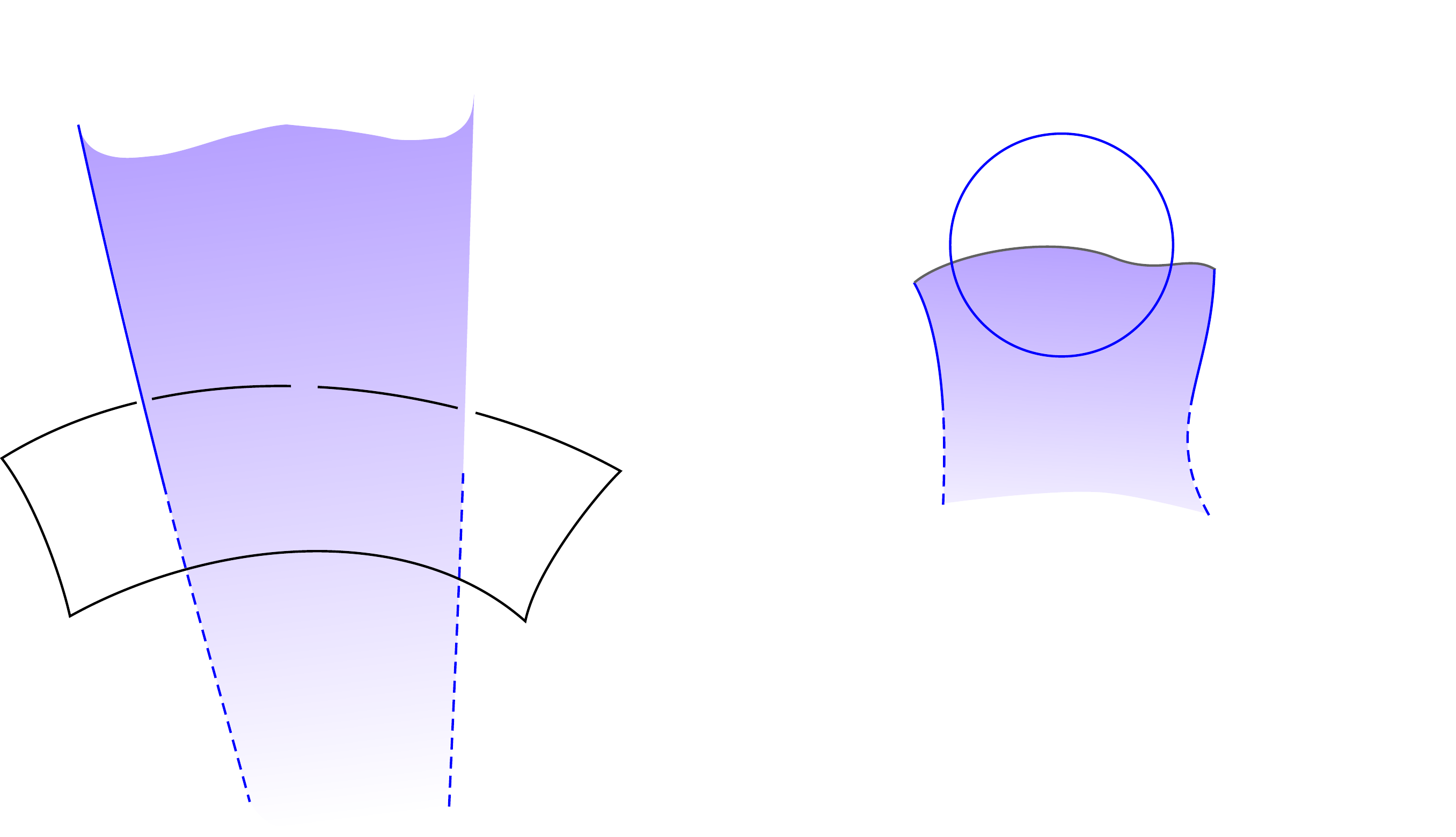

        \caption{The open sets in $\nu$ and in $M$}
        \label{fig:SeIsDense}
    \end{figure}

    If (\hyperref[thm:separatingSetDense:case:neverInjective]{b}) appears, then for each $k$ we can pick $\mathbf{u}_k^1 \ne \mathbf{u}_k^2 \in \overline{\mathcal{U}}_k$, so that 
    \[\exp^\nu(\mathbf{u}_k^1) = z_k = \exp^\nu(\mathbf{u}_k^2).\]
    For $j = 1,2$, let us write, $\mathbf{u}^j_k = c^j_k \mathbf{v}^j_k$, for $\mathbf{v}^j_k \in \overline{V(\mathbf{v}, \delta_k)} \cap S(\nu)$, and $c^j_k \le \rho(\mathbf{v}^j_k)$. But $z_k \in \mathrm{Se}(N)$, and so $c^j_k = \rho(\mathbf{v}^j_k)$.  As $\cap_k \overline{V(\mathbf{v}, \delta_k)} = \left\{ \mathbf{\mathbf{v}} \right\}$, we have $\mathbf{v}_k^j \rightarrow  \mathbf{v}$. By the continuity of $\rho$, we have $c^j_k = \rho(\mathbf{v}^j_k) \rightarrow \rho(\mathbf{v}) = \ell$. Consequently, $\mathbf{u}^j_k = c^j_k \mathbf{v}^j_k \rightarrow \ell \mathbf{v} = \mathbf{u}$, for $j = 1,2$. Then, it follows from the continuity of $\exp^\nu$ that $z_k = \exp^\nu(\mathbf{v}^j_k) \rightarrow  \exp^\nu(\mathbf{u}) = q$. As $z_k \in \mathrm{Se}(N)$, we have $q$ is in the closure of $\mathrm{Se}(N)$, proving that $\mathrm{Cu}(N) \subset \overline{\mathrm{Se}(N)}$.

    Now, suppose $q \in \overline{\mathrm{Se}(N)} \setminus \mathrm{Se}(N)$. Consider the unique $N$-segment $\gamma : [0, \ell] \rightarrow M$ joining $N$ to $q = \gamma(\ell)$, where $\ell = d(N, q)$. Let $q_i \in \mathrm{Se}(N)$ so that $q_i \rightarrow q$. Consider two distinct $N$-segments, $\eta^1_i, \eta^2_i : [0, \ell_i] \rightarrow M$, joining $N$ to $q_i$. Suppose $\eta^j_i(t) = \exp^\nu(t \mathbf{v}^j_i)$, where $\ell_i = d(N, q_i)$ and $\mathbf{v}_i^1 \ne \mathbf{v}_i^2 \in S(\nu)$. Now, $\ell_i = d(N, q_i) \rightarrow d(N, q) = \ell$. Hence, by \autoref{prop:existenceOfNSegements}, we have a subsequence $\eta^j_i \rightarrow \eta^j$, where $\eta^j : [0, \ell] \rightarrow M$ is an $N$-segment joining $N$ to $q$. But as $q \not \in \mathrm{Se}(N)$, we have $\eta^1 = \gamma = \eta^2$. In particular, we get $\ell_i \mathbf{v}_i^j \rightarrow \ell \mathbf{v}$ for $j = 1, 2$. But then, $\exp^\nu$ is not injective in any neighborhood of $\ell \mathbf{v}$, and hence by the inverse function theorem, $d\left( \exp^\nu \right)|_{\ell \mathbf{v}}$ cannot have full rank. Consequently, $q = \exp^\nu(\ell \mathbf{v})$ is a focal point of $N$ along $\gamma$. Since $\gamma$ is an $N$-segment, it follows from \autoref{prop:uniquenessOfGeodesicBeforeCutPoint} that $q \in \mathrm{Cu}(N)$. Thus, $\overline{\mathrm{Se}(N)} \subset \mathrm{Cu}(N)$, completing the proof.
\end{proof}

\subsection{The Topology of the Complement of Cut Locus} Let us consider the subset 
\[\mathcal{I}(N) = \left\{ t\mathbf{v} \;\middle|\; 0\le t < \rho(\mathbf{v}), \mathbf{v} \in S(\nu) \right\} \subset \nu,\]
and denote $\hat{\mathcal{I}}(N) = \mathcal{I}(N) \cap \hat{\nu}$. It is clear that $\mathcal{I}(N) \cap \widetilde{\mathrm{Cu}}(N) = \emptyset$, and $\exp^\nu(\mathcal{I}(N)) \cap \mathrm{Cu}(N) = \emptyset$. Observe that $\mathcal{I}(N)$ is \emph{star-shaped}, i.e., for any $\mathbf{v} \in \mathcal{I}(N)$ we have $t\mathbf{v} \in \mathcal{I}(N)$ for $0 \le t \le 1$.

\begin{thm}\label{thm:normalExponetnialDiffeo}
    Let $N$ be a closed submanifold of a forward complete Finsler manifold $(M, F)$, and additionally (\hyperref[eq:hypothesisH]{H}) holds.  Then, the normal exponential map $\exp^\nu$ induces 
    \begin{itemize}
        \item a $C^1$-diffeomorphism of $\mathcal{I}(N)$ onto $M \setminus \mathrm{Cu}(N)$, and
        \item a $C^\infty$ diffeomorphism of $\hat{\mathcal{I}}(N)$ onto $M \setminus (N\cup \mathrm{Cu}(N))$.
    \end{itemize} 
\end{thm}
\begin{proof}
    Let $q \in M \setminus \left( N \cup \mathrm{Cu}(N) \right)$. Then, it follows from \autoref{prop:existenceOfNSegements} and \autoref{prop:separtingSetContainedInCutLocus} that, there exists a unique $N$-segment $\gamma : [0, \ell] \rightarrow M$ joining $N$ to $q = \gamma(\ell)$, given by $\gamma(t) = \exp^\nu(t \mathbf{v})$, where $\ell = d(N, q)$ and $\mathbf{v} = \dot \gamma(0) \in S(\nu)$. In particular, we have a map $\Phi: M \setminus \mathrm{Cu}(N) \rightarrow \mathcal{I}(N)$ defined as follows:
    \[\Phi(q) = \begin{cases}
        \ell \mathbf{v}, & q\not\in N \cup \mathrm{Cu}(N), \text{ and } q = \exp^\nu(\ell \mathbf{v})\\
        0 \in T_q N, & q \in N
    \end{cases}\]
    By construction, $\Phi$ is the required inverse map to $\exp^\nu$. We show that $\Phi$ is $C^1$ on $M \setminus \mathrm{Cu}(N)$, and is $C^\infty$ when restricted to $M\setminus (N \cup \mathrm{Cu}(N))$.

    Let $q \not \in N \cup \mathrm{Cu}(N)$, and $\mathbf{u} = \ell \mathbf{v} = \Phi(q) \in \hat{\mathcal{I}}(N)$. It follows from \autoref{lemma:beyondFocalPoint} that $q$ is not a focal point of $N$ along the $N$-geodesic $\gamma(t) = \exp^\nu(t \mathbf{v})$. In particular, $d\left( \exp^\nu \right)|_{\mathbf{u}}$ is nondegenerate, and hence, it is a local diffeomorphism. By the implicit function theorem, $\Phi$ is then $C^\infty$ in a neighborhood of $q$. Thus, $\Phi$ is $C^\infty$ on $M \setminus \left( N\cup\mathrm{Cu}(N) \right)$. Since $\exp^\nu$ is $C^1$ on the $0$ section of $\nu$, same argument as above gives that $\Phi$ is $C^1$ on $M \setminus \mathrm{Cu}(N)$. This completes the proof.
\end{proof}

As a consequence, we now have the following.
\begin{thm}\label{thm:deformation}
    Suppose $N$ is as in \autoref{thm:normalExponetnialDiffeo}. Then,
    \begin{enumerate}
        \item\label{thm:deformation:1} $M\setminus \mathrm{Cu}(N)$ strongly deformation retracts onto $N$, and
        \item\label{thm:deformation:2} $M \setminus N$ strongly deformation retracts onto $\mathrm{Cu}(N)$.
    \end{enumerate}  
\end{thm}
\begin{proof}
    Consider the linear homotopy $L_s : TM \rightarrow  TM$ given by $L_s(\mathbf{v}) = (1 - s)\mathbf{v}$, for $0 \le s \le 1$. By construction, $\mathcal{I}(N)$ is star-shaped, and thus, $L_s$ restricts to a homotopy of $\nu$ as well. Indeed, 
    \[L_0|_{\mathcal{I}(N)} = \mathrm{Id}_{\mathcal{I}(N)}, \quad L_1|_{\mathcal{I}(N)} = \pi|_{\mathcal{I}(N)},\]
    where $\pi : \nu \rightarrow N$ is the projection. Next, we define the map $h : I \times (M \setminus \mathrm{Cu}(N)) \rightarrow (M \setminus \mathrm{Cu}(N))$ by
    \[h(s, q) \coloneqq \exp^\nu \circ L_s \circ \left( \exp^\nu \right)^{-1}(q).\]
    It follows from \autoref{thm:normalExponetnialDiffeo} that $h$ is a well-defined homotopy. Clearly, $h_0 = h(0, \_)$ is the identity on $M \setminus \mathrm{Cu}(N)$, whereas $h_1 = h(1, \_)$ is contained in $N$. Furthermore, $h_s = h(s, \_)$ is the identity on $N$ for all $0 \le s \le 1$, and $h_1 \circ \iota_N = \mathrm{Id}_N$, where $\iota_N : N \hookrightarrow M$ is the inclusion. Thus, $M \setminus \mathrm{Cu}(N)$ strongly deformation retracts onto $N$, which proves (\hyperref[thm:deformation:1]{1}).
    
    In order to prove (\hyperref[thm:deformation:2]{2}) consider the homotopy $H : I \times (M \setminus N) \rightarrow M \setminus N$ defined as 
    \[H(s, q) \coloneqq
    \begin{cases}
        \exp^\nu\left[ \left( s \rho\left( \frac{\mathbf{v}}{F(\mathbf{v})} \right) + (1 - s) F(\mathbf{v}) \right) \frac{\mathbf{v}}{F(\mathbf{v})} \right], & q \in M \setminus \left( N \cup \mathrm{Cu}(N) \right), 
        \\ &\mathbf{v} \coloneqq \left( \exp^\nu \right)^{-1}(q) 
        \\[1em]
        q, & q \in \mathrm{Cu}(N).
    \end{cases}\]
    Continuity of $H$ follows from \autoref{thm:rhoContinuous} and \autoref{thm:normalExponetnialDiffeo}. Clearly, $H_0 = H(0, \_)$ is identity on $M\setminus N$, whereas $H_1 = H(1, \_)$ is contained in $\mathrm{Cu}(N)$. Furthermore, for all $0 \le s \le 1$, we have $H_s = H(s,\_)$ is identity on $\mathrm{Cu}(N)$, and $H_1 \circ \iota_{\mathrm{Cu}(N)} = \mathrm{Id}|_{\mathrm{Cu}(N)}$, where $\iota_{\mathrm{Cu}(N)} : \mathrm{Cu}(N) \hookrightarrow M$ is the inclusion.  Thus, $M \setminus N$ strongly deformation retracts onto $\mathrm{Cu}(N)$.
\end{proof}

\begin{obs}\label{obs:compactnessAndTangentCutLocus}
    Suppose $N$ is a closed submanifold of $M$. We make a few easy observations. 
    \begin{enumerate}
        \item\label{obs:compactness:MIsImage} $M = \exp^\nu \left( \mathcal{I}(N) \cup \widetilde{\mathrm{Cu}}(N) \right)$.
        
        \item\label{obs:compactness:MCompactImpliesRhoFinite} If $M$ is compact, then we can bound $\rho$ by the number $\sup_{q \in M} d(N, q) < \infty$. The converse may not be true. As an example, consider the infinite cylinder $M = S^1 \times \mathbb{R}$  with the standard metric, and $N = \left\{ 1 \right\} \times \mathbb{R}$ be the infinite line. Then, $\rho = \frac{1}{2}$ everywhere.

        \item\label{obs:compactness:NCompactRhoFiniteIMpliesMCompact} If $N$ is compact, then finiteness of $\rho$ implies the compactness of $M$. Indeed, consider the unit normal disc bundle of $N$ 
        \[D(\nu) \coloneqq \left\{ t\mathbf{v} \;\middle|\; 0 \le t \le 1, \; \mathbf{v}\in S(\nu) \right\},\]
        which is compact as $N$ is so. Clearly, $D(\nu) = \mathcal{I}(N) \cup \widetilde{\mathrm{Cu}}(N)$, and hence by (\hyperref[obs:compactness:MIsImage]{1}), $M$ is compact.
        
        \item\label{obs:compactness:MIsCompactIfLambdaFinite} If $\lambda < \infty$ and $N$ is compact, then $M$ is compact by (\hyperref[obs:compactness:NCompactRhoFiniteIMpliesMCompact]{3}), since $\rho \le \lambda$ by \autoref{cor:cutTimeLessThanFocalTime}. The converse may not be true. Consider $N = S^1\times \left\{ 1 \right\}$ in $M = S^1 \times S^1$ with the flat metric, which admits no focal locus, and thus $\lambda = \infty$.
    \end{enumerate}
\end{obs}

\subsection{Distance Function from \texorpdfstring{$N$}{N}}
\begin{thm}\label{thm:distanceSquareDiff}
    Let $N$ be a closed submanifold, and (\hyperref[eq:hypothesisH]{H}) holds. Let $f: M \rightarrow \mathbb{R}$ be the distance squared function $f(q) = d(N, q)^2$. Then,
    \begin{enumerate}
        \item $f$ is $C^1$ precisely on $M \setminus \mathrm{Se}(N)$, and $f$ is $C^\infty$ on $M \setminus (N \cup \mathrm{Cu}(N))$.
        \item For $q \not \in \mathrm{Se}(N)$ and $X \in T_q M$, we have $df(X) = 2\ell g_{\dot\gamma(\ell)}(\dot\gamma(\ell), X)$, where $\gamma : [0, \ell] \rightarrow  M$ is the unique $N$-segment joining $N$ to $q = \gamma(\ell)$ and $\ell = d(N, q)$.
        \item For $q \in \mathrm{Se}(N)$, $f$ fails to be differentiable precisely along those $X \in T_q M$ for which there exists an $N$-segment $\gamma : [0,\ell] \rightarrow M$ joining $N$ to $q$ with $\dot\gamma(\ell) = X$.
    \end{enumerate}
\end{thm}

The proof is essentially the same as the one given in \cite{SaTa16}, where the authors proved that the distance function from an arbitrary closed set $N$ in a Finsler manifold is $C^1$ precisely in the complement of $N \cup \mathrm{Cu}(N)$. Recall the following.

\begin{lemma}\cite[Proposition 2.2]{SaTa16}\label{lemma:geometricFirstVariation}
    Consider the following setup.
    \begin{itemize}
        \item $\gamma : [0,\ell] \rightarrow M$ is an $N$-segment, joining $p = \gamma(0)$ to $q = \gamma(\ell)$, where $\ell = d(N, q)$.
        \item $q_i \rightarrow q$ is a converging sequence in $M$.
        \item $\gamma_i : [0,\ell_i] \rightarrow M$ is an $N$-segment, joining $p_i = \gamma_i(0)$ to $q_i = \gamma_i(\ell_i)$, where $\ell_i = d(N, q_i)$, and $\mathbf{v}_i = \dot\gamma_i(0)$.
        \item $\sigma_i : [0, T_i] \rightarrow M$ is an $N$-segment, joining $q = \sigma_i(0)$ to $q_i = \sigma_i(T_i)$, where $T_i = d(q, q_i)$, and $\mathbf{w}_i = \dot\sigma_i(0)$.
    \end{itemize}
    Suppose, the limit $\mathbf{w} = \lim_i \mathbf{w}_i$ exists, and $\gamma_i$ converges uniformly to $\gamma$. Then, the limit 
    \[\lim_{i\rightarrow \infty} \frac{d(N, q_i) - d(N, q)}{d(q, q_i)}\]
    exists, and equals $g_{\dot\gamma(\ell)}\left( \dot\gamma(\ell), \mathbf{w} \right)$.
\end{lemma}

\begin{proof}[Proof of \autoref{thm:distanceSquareDiff}]
    We have from \autoref{thm:normalExponetnialDiffeo}, $\exp^\nu$ restricts to a $C^\infty$-diffeomorphism $\hat{\mathcal{I}}(N) \rightarrow M \setminus(N \cup \mathrm{Cu}(N))$. For each $q \in M \setminus (N \cup \mathrm{Cu}(N))$, we have the unique $N$-segment $\gamma_{\mathbf{v}}$ joining $N$ to $q$, where $\mathbf{v}=(\exp^\nu)^{-1}(q)$, and furthermore $d(N, q) = L(\gamma_\mathbf{v})=F(\mathbf{v})$. Thus, we have the identity 
    \[f|_{M\setminus (N \cup \mathrm{Cu}(N))} = \left( F \circ (\exp^\nu)^{-1} \right)^2,\]
    which proves the smoothness of $f$ on the complement of $N \cup \mathrm{Cu}(N)$.

    Now, let $q \not\in \mathrm{Se}(N)$. Then, there exists a unique $N$-segment $\gamma : [0, \ell] \rightarrow  M$ joining $p = \gamma(0) \in N$ to $q = \gamma(\ell)$, where $\ell = d(N, q)$ and $\mathbf{v} = \dot\gamma(0)\in S(\nu)$. Clearly, $f(q) = \ell^2 = (L(\gamma))^2$. Let $X \in T_q M$, and consider a minimizer $\sigma : [0, \epsilon) \rightarrow M$ such that $\sigma(0) = q, \dot\sigma(0) = X$. Let $s_i \rightarrow 0$ be a sequence in the domain of $\sigma$. Suppose, $\gamma_i : [0, \ell_i] \rightarrow M$ are $N$-segments, joining $p_i = \gamma_i(0) \in N$ to $q_i = \sigma(s_i) = \gamma_i(\ell_i)$, where $\ell_i = d(N, q_i) = L(\gamma_i)$. As $s_i \rightarrow 0$, we have $q_i = \sigma(s_i) \rightarrow  \sigma(0) = q$, and $\ell_i \rightarrow \ell$. Set $\mathbf{v}_i = \dot\gamma_i(0) \in S(\nu)$. It follows from \autoref{prop:existenceOfNSegements}, a subsequence $\gamma_{i} \rightarrow \gamma$, and $(p_i, \mathbf{v}_i) \rightarrow (p, \mathbf{v})$. In particular, $\dot\gamma(\ell) = \lim_i \dot\gamma_i(\ell_i)$. Then, it follows from \autoref{lemma:geometricFirstVariation} that 
    \begin{align*}
        \lim_{i\rightarrow \infty} \frac{f(\sigma(s_i)) - f(\sigma(0))}{s_i}
        &= \lim_{i\rightarrow \infty} \frac{d(N, q_i)^2 - d(N,q)^2}{s_i} \\
        &= \lim_{i \rightarrow \infty} \left( d(N, q_i) + d(N, q) \right) \frac{d(N, q_i) - d(N,q) }{d(q_i, q)}  \\
        &=  \left( d(N, q) + d(N, q) \right) g_{\dot \gamma(\ell)} \left( \dot \gamma(\ell), X \right) \\
        &= 2 \ell g_{\dot \gamma(\ell)} \left( \dot \gamma(\ell), X \right).
    \end{align*}
    Thus, $f$ is differentiable from the left at $q$. The backward version of \autoref{lemma:geometricFirstVariation} \cite[Proposition 2.3]{SaTa16} then gives that $f$ is differentiable from the right as well. Consequently, $f$ is differentiable at $q$ with differential $df(X) = 2 \ell g_{\dot \gamma(\ell)}\left( \dot \gamma(\ell), X \right)$. A careful examination of the proof of \autoref{lemma:geometricFirstVariation} (and its backward version) from \cite{SaTa16}, reveals that $f$ fails to be differentiable at points of $\mathrm{Se}(N)$ precisely along those directions which is the terminal velocity of some $N$-segment. This concludes the proof.
\end{proof}

\subsection{Existence of \texorpdfstring{$N$}{N}-Geodesic Loops}\label{subsec:NGeodesicLoop}
In this subsection, we generalize a result following \cite{Innami2012}. The proof is in the same vein, with suitable modifications needed to deal with the asymmetric nature of Finsler distance.

\begin{thm}\label{thm:atLeastTwoNGeodesics}
    Let $N$ be a closed submanifold of a forward complete Finsler manifold $(M, F)$. Suppose (\hyperref[eq:hypothesisH]{H}) holds, that is, either (a) $N$ is compact, or (b) $F$ is also backward complete. Then one of the following is true.
    \begin{itemize}
        \item There are points in $\mathrm{Cu}(N)$ which are also focal points of $N$.
        \item For each $q \in M$, there exist at least two $N$-geodesics joining $N$ to $q$, one of which crosses the cut locus $\mathrm{Cu}(N)$. For $q\in N$, we allow the constant curve at $q$ to be considered as an $N$-geodesic.
    \end{itemize}
\end{thm}

Note that, among the hypotheses of the above theorem, both (a) and (b) are true if $M$ is compact, whereas (c) follows if $F$ (which is forward complete) is assumed to be reversible. As an application of \autoref{thm:atLeastTwoNGeodesics}, in the absence of any focal cut points, we can get an $N$-geodesic joining $N$ to some point $q \in N$, although the terminal velocity need not be in the normal cone $\nu_q$. 

\begin{defn}\label{defn:NGeodesicLoop}
    An $N$-geodesic $\gamma : [0, \ell] \rightarrow M$ is called an \emph{$N$-geodesic loop} if $\gamma$ connects $p = \gamma(0) \in N$ to $q = \gamma(\ell) \in N$, and $\dot \gamma(\ell) \in \nu_q(N)$.
\end{defn}

Such $N$-geodesics were considered by Omori \cite{Omo68} in the Riemannian setting. We have the following result, originally due to Klingenberg \cite[Lemma 1]{Kli59} in the Riemannian case for a point, and later generalized in \cite{Xu2015}.

\begin{thm}\label{thm:NGeodesicLoop}
    Let $N$ be a closed submanifold of a forward complete Finsler manifolds $(M, F)$, where $F$ is \emph{reversible}. Suppose $x_0 \in \mathrm{Cu}(N)$ is a global minima of the function $x \mapsto d(N,x)$ on $\mathrm{Cu}(N)$. Then one of the following holds.
    \begin{itemize}
        \item $x_0$ is a focal point of $N$.
        \item There exists a unique $N$-geodesic loop, which crosses $\mathrm{Cu}(N)$ at $x_0$, at the midpoint.
    \end{itemize}
    If $x_0$ is only a local minima, then the above holds, provided $\mathrm{Cu}(N)$ is compact.
\end{thm}

Note that in the above theorem, condition (\hyperref[eq:hypothesisH]{H}) is satisfied as reversibility of $F$ implies that $F$ is backward complete as well. In order to prove these two theorems, for any $q \in M \setminus \mathrm{Cu}(N)$, let us consider the function 
\begin{align*}
    \mathcal{M} = \mathcal{M}_q : M &\rightarrow [0,\infty) \\
    x &\mapsto d(N, x) + d(x, q)
\end{align*}
Suppose there exists some $x_0 \in \mathrm{Cu}(N)$ so that $\mathcal{M}|_{\mathrm{Cu}(N)}$ attains a global minima at $x_0$, i.e.,
\[\mathcal{M}(x_0) = \min_{x \in \mathrm{Cu}(N)} \mathcal{M}(x).\]
We have the following lemma, which is quite interesting in itself.

\begin{lemma}\label{lemma:twoNSegments}
    With the above notation, further assume that $x_0$ is not a focal point of $N$. We then have the following dichotomy.
    \begin{enumerate}
        \item \label{lemma:twoNSegments:1} $q \not\in \mathrm{Cu}(x_0)$. Then, there exist exactly two $N$-segments, say, $\gamma_1,\gamma_2$, joining $N$ to $x_0$, and a unique minimizer, say $\eta$, joining $x_0$ to $q$. Furthermore, exactly one of $\gamma_1 \cup \eta$ and $\gamma_2 \cup \eta$ is an $N$-geodesic joining $N$ to $q$, and intersecting $\mathrm{Cu}(N)$ exactly at $x_0$.
        
        \item \label{lemma:twoNSegments:2} $q \in \mathrm{Cu}(x_0)$. Then, there exist exactly two $N$-segments, say, $\gamma_1, \gamma_2$, joining $N$ to $x_0$, and at most two minimizers, say, $\eta_1, \eta_2$, joining $x_0$ to $q$. Furthermore, if $\eta_1, \eta_2$ are distinct, then among the four curves $\left\{ \gamma_i \cup \eta_j, i,j = 1,2 \right\}$, exactly one pair
        \[
            \left\{ \gamma_1 \cup \eta_1, \gamma_2 \cup \eta_2 \right\} \text{ or } \left\{ \gamma_1 \cup \eta_2, \gamma_2 \cup \eta_1 \right\}
        \]
        are smooth curves, joining $N$ to $q$, crossing $\mathrm{Cu}(N)$ exactly at $x_0$. The other two curves are broken at $x_0$. If $\eta_1 = \eta_2 = \eta$, then exactly one of $\left\{ \gamma_1 \cup \eta, \gamma_2 \cup \eta \right\}$ is an $N$-geodesic joining $N$ to $q$, crossing $\mathrm{Cu}(N)$ exactly at $x_0$, the other one being broken at $x_0$.
        
        \item[(2')] \label{lemma:twoNSegments:2prime} For $q \in \mathrm{Cu}(x_0)$, if we furthermore assume that $F$ is a reversible Finsler metric, then exactly two minimizers exist from $x_0$ to $q$, and the rest of (\hyperref[lemma:twoNSegments:2]{2}) holds.
    \end{enumerate}
\end{lemma}
\begin{proof}
    \textbf{Case (\hyperref[lemma:twoNSegments:1]{1}):} Suppose $q \not \in \mathrm{Cu}(x_0)$. Then, it follows from \autoref{prop:uniquenessOfGeodesicBeforeCutPoint} that there exists a unique minimizer $\eta : [0, a] \rightarrow M$ joining $x_0$ to $q$, where $a = d(x_0, q)$. Let us show that $\eta$ only intersects $\mathrm{Cu}(N)$ at the initial point $x_0$. Indeed, for any $0 < t \le a$ observe that any curve joining $N$ to $\eta(t)$ and passing through $x_0 \in \mathrm{Cu}(N)$ cannot be an $N$-segment, and thus
    \[d(N, \eta(t)) < d(N, x_0) + d(x_0, \eta(t)).\]
    Then we have,
    \begin{align*}
        \mathcal{M}(\eta(t)) &= d(N, \eta(t)) + d(\eta(t), q) \\
        &< d(N, x_0) + \underbrace{d(x_0, \eta(t)) + d(\eta(t), q)}_{\text{$x_0, \eta(t), q$ lies on the minimizer $\eta$}} \\
        &= d(N, x_0) + d(x_0, q) \\
        &= \mathcal{M}(x_0).
    \end{align*}
    As $x_0$ is a global minima of $\mathcal{M}|_{\mathrm{Cu}(N)}$, we must have $\eta(t) \not \in \mathrm{Cu}(N)$ for $0 < t \le a$. 
    
    Now, since $x_0 \in \mathrm{Cu}(N)$ is not a focal point of $N$, we have $x_0 \in \mathrm{Se}(N)$, and hence there are at least $2$ unique $N$-segments joining $N$ to $x_0$. If possible, assume that we have at least three $N$-segments joining $N$ to $x_0$, say, $\gamma_1, \gamma_2, \gamma_3 : [0, \ell] \rightarrow M$, where $\ell = d(N, x_0)$ (\autoref{fig:threeGeodesics}). Since they are distinct, their terminal velocities $\dot\gamma_i(\ell)$ are distinct as well. Consequently, \emph{at most} one of $\gamma_i \cup \eta$ can be smooth at $x_0$. Without loss of generality, assume that both $\gamma_1 \cup \eta$ and $\gamma_2 \cup \eta$ are broken at $x_0$.
    \begin{figure}[H]
        \centering
    \def\svgwidth{0.5\columnwidth}
    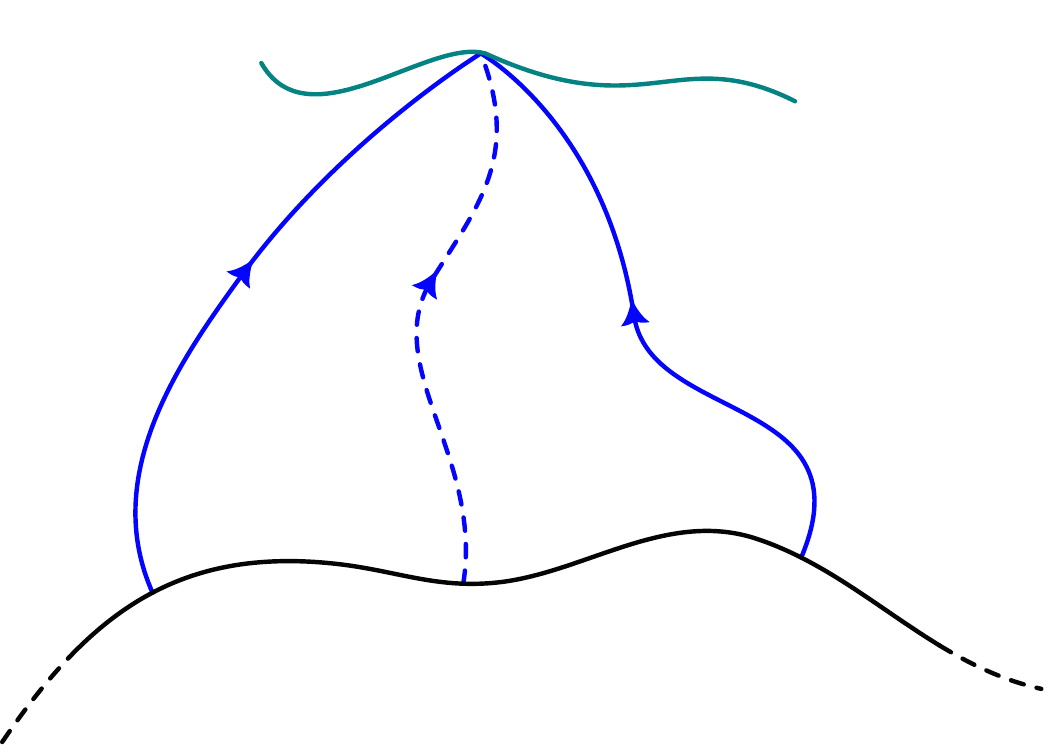

        \caption{Case (\hyperref[lemma:twoNSegments:1]{1})}
        \label{fig:threeGeodesics}
    \end{figure}

    Since $\gamma_1$ is in particular a minimizer, $\gamma_1|_{[t, \ell]}$ is a minimizer for any $0 < t < \ell$ (\autoref{prop:globalMinimizerIsLocalMinimizer}). Then, for $y = \gamma_1(t)$, it follows from the brokenness of $\gamma_1 \cup \eta$ that 
    \[d(y, q) < d(y, x_0) + d(x_0, q).\]
    Hence,
    \[\mathcal{M}(y) = d(N, y) + d(y, q) < d(N,y) + d(y, x_0) + d(x_0, q) = d(N, x_0) + d(x_0, q) = \mathcal{M}(x_0).\]
    Now, let $\zeta : [0,b] \rightarrow M$ be a minimizer joining $y$ to $q$, where $b = d(y, q)$. We claim that $\zeta$ does not intersect $\mathrm{Cu}(N)$. If not, we must have some $0 < s < b$ so that $x_0^\prime = \zeta(s) \in \mathrm{Cu}(N)$ (see \autoref{fig:reflects-case-1}).
    \begin{figure}[H]
        \centering
    \def\svgwidth{0.5\columnwidth}
    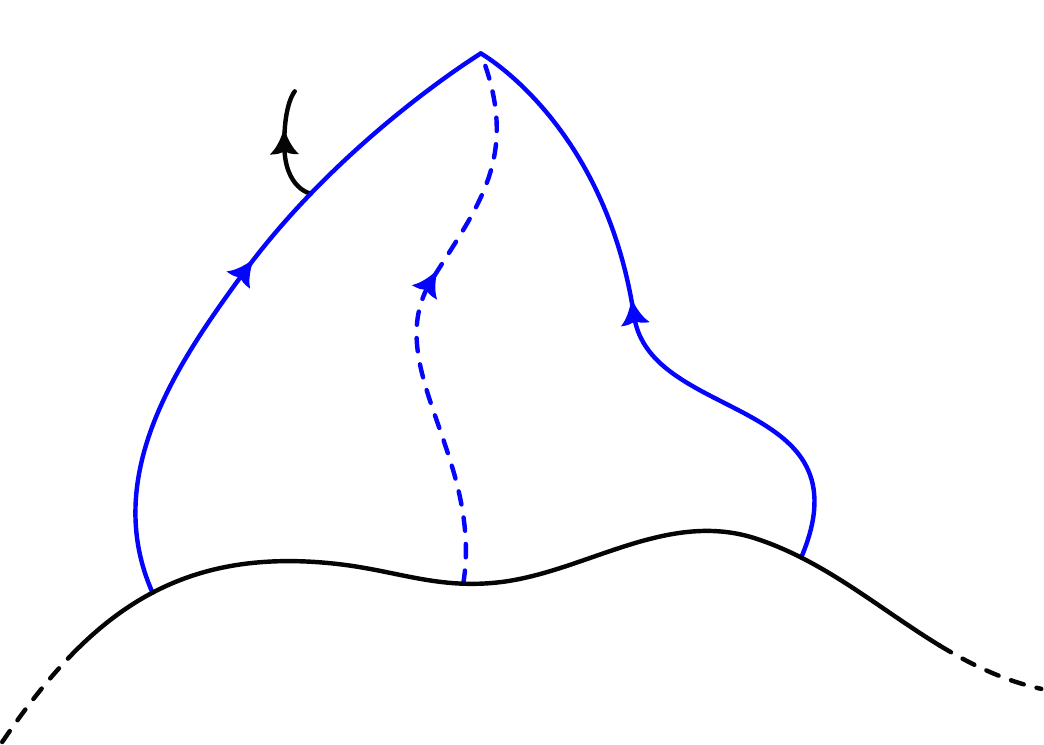

        \caption{Case (\hyperref[lemma:twoNSegments:1]{1}) : Minimizer $\zeta$ joining $y$ to $q$, crossing $\mathrm{Cu}(N)$}
        \label{fig:reflects-case-1}
    \end{figure}
    We have
    \begin{align*}
        \mathcal{M}(x_0) \le \mathcal{M}(x_0^\prime) 
        &= d(N, x_0^\prime) + d(x_0^\prime, q) \\
        &\le d(N, y) + \underbrace{d(y, x_0^\prime) + d(x_0^\prime,q)}_{\text{$y,x_0^\prime,q$ are on the minimizer $\zeta$}} \\
        &= d(N,y) + d(y,q) = \mathcal{M}(y) < \mathcal{M}(x_0),
    \end{align*}
    which is a contradiction. Thus, $\zeta$ does not intersect $\mathrm{Cu}(N)$. Similarly, for any $y = \gamma_2(t)$, we have that any minimizer joining $y$ to $q$ does not intersect $\mathrm{Cu}(N)$ (\autoref{fig:reflects-case-2}).

    \begin{figure}[H]
        \centering
    \def\svgwidth{0.5\columnwidth}
    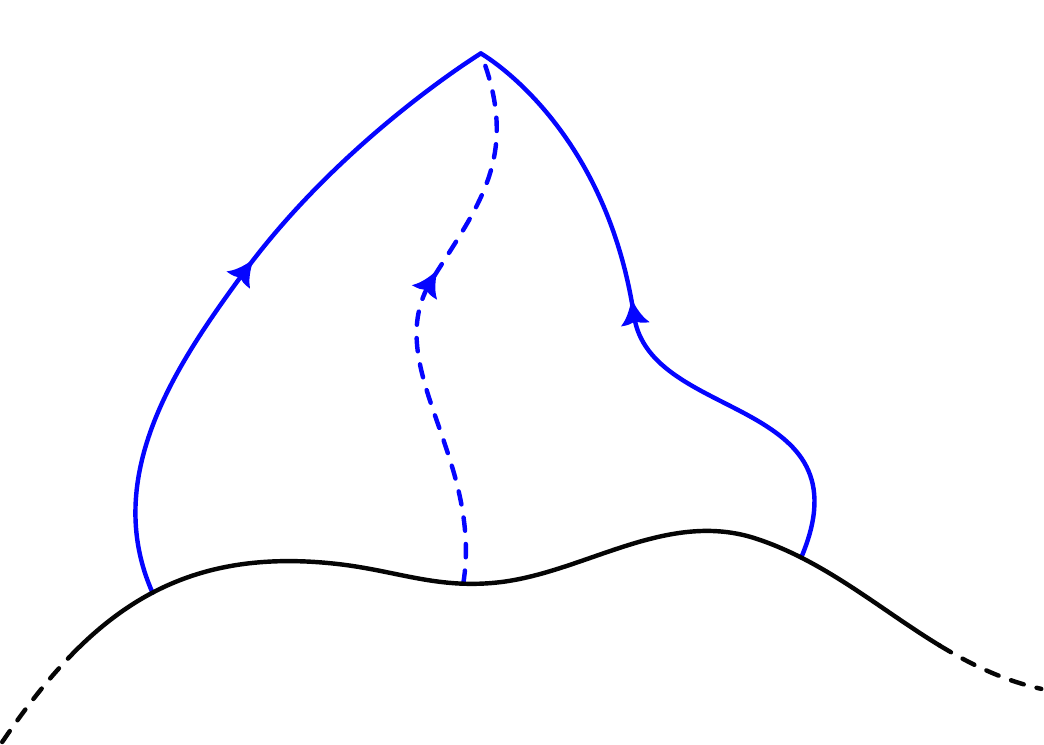

        \caption{Case (\hyperref[lemma:twoNSegments:1]{1}) : Minimizer between $y_{ij}$ and $q$}
        \label{fig:reflects-case-2}
    \end{figure}
    Pick some sequence $t_j \rightarrow \ell$ and set $y_{ij} = \gamma_i(t_j), i=1,2$. Then, $y_{ij} \rightarrow \gamma_i(\ell) = x_0$ for $i=1,2$. Consider minimizers $\zeta_{ij} : [0, a_{ij}] \rightarrow M$ joining $y_{ij}$ to $q$, where $a_{ij} = d(y_{ij}, q)$. Clearly, $a_{ij} = d(y_{ij}, q) \rightarrow d(x_0, q) = a$, and hence by \autoref{lemma:convergentSubsequenceNSegment}, a subsequence $\zeta_{ij}$ converges uniformly to a minimizer joining $x_0$ to $q$, which is necessarily $\eta$ by uniqueness. Reparametrizing, we have the curves $\hat{\zeta}_{ij} : [0,a] \rightarrow M$ given by $\hat{\zeta}_{ij}(t) = \zeta_{ij}\left( \frac{a_{ij}}{a} t \right)$, which converges uniformly to $\eta$. Note that, the family $\left\{ \hat{\zeta}_{ij} \right\}_j$ is uniformly Lipschitz continuous for $i =1, 2$. Indeed, as $a_{ij} \rightarrow a$, we have a number $a_0$ such that $\sup_j \frac{a_{ij}}{a} \le a_0$. As $\zeta_{ij}$ is a minimizer, for $0 \le t_1 \le t_2 \le a$ we get $d(\hat{\zeta}_{ij}(t_1), \hat{\zeta}_{ij}(t_2)) = \frac{a_{ij}}{a}(t_2 - t_1) \le a_0 |t_2 - t_1|$, and thus, $a_0$ is a uniform Lipschitz constant for both the families $\left\{ \hat{\zeta}_{ij} \right\}_j$. As $\zeta_{ij}$ does not intersect $\mathrm{Cu}(N)$, it follows from \autoref{thm:normalExponetnialDiffeo} that there exists unique lift $\tilde{\zeta}_{ij} \coloneqq \left( \exp^\nu|_{\mathcal{I}(N)} \right)^{-1} \circ \hat{\zeta}_{ij} : [0, a] \rightarrow \mathcal{I}(N)$ joining $\tilde{\mathbf{y}}_{ij} \coloneqq \tilde{\zeta}_{ij}(0)$ to $\tilde{\mathbf{q}} \coloneqq  \tilde{\zeta}_{ij}(a)$. Clearly, $\exp^\nu(\tilde{\mathbf{y}}_{ij}) = y_{ij}$ and $\exp^\nu(\tilde{\mathbf{q}}) = q$. In particular, $\tilde{\mathbf{y}}_{ij} = t_j \dot\gamma_i(0) \rightarrow \ell \dot\gamma_i(0)$. We claim that for $i = 1, 2$, a subsequence of $\tilde{\zeta}_{ij}$ converges to some $\tilde{\zeta}_i$, which must then join $\mathbf{w}_i \coloneqq  \ell \dot\gamma_i(0) \in \left( \exp^\nu \right)^{-1}(x_0)$ to $\tilde{\mathbf{q}}$. \vspace{.6em}

    Since $x_0$ is not a focal point of $N$, for $i=1,2$ we have that $\exp^\nu$ is a diffeomorphism restricted to an open neighborhood, say, $\mathcal{W}_i \subset \hat{\nu}$ of $\mathbf{w}_i$. We pick $\mathcal{W}_i$ sufficiently small so that $\mathcal{W}_1 \cap \mathcal{W}_2 = \emptyset$. Consider the open set $\mathcal{U} \coloneqq \hat{\mathcal{I}}(N) \cup \mathcal{W}_1 \cup \mathcal{W}_2 \subset \hat{\nu}$, and define the map $\hat{F} : T\mathcal{U} \rightarrow [0, \infty)$ by 
    \[\hat{F}_{\mathbf{v}}(\mathbf{x}) \coloneqq F_{\exp^\nu(\mathbf{v})} \left( d_{\mathbf{v}} \left( \exp^\nu|_{\hat{\nu}} \right) (\mathbf{x}) \right), \quad \mathbf{x} \in T_{\mathbf{v}} \  \mathcal{U}, \; \mathbf{v} \in \mathcal{U}.\]
    Since $\exp^\nu|_{\mathcal{U}}$ is an immersion, one can easily check that $\hat{F}$ is a Finsler metric on $\mathcal{U}$, induced from $F$ by $\exp^\nu|_{\mathcal{U}}$. Furthermore, $\tilde{\zeta}_{ij}$ is a minimizer with respect to $\hat{F}$. Indeed, we observe that
    \begin{align*}
        L(\tilde{\zeta}_{ij}) &\ge d\left( \tilde{\zeta}_{ij}(0), \tilde{\zeta}_{ij}(a) \right) \\
        &= \inf \left\{ L(\tau) \;\middle|\; \tau : [0, a] \rightarrow \mathcal{U}, \, \tau(0) = \tilde{\zeta}_{ij}(0), \, \tau(a) = \tilde{\zeta}_{ij}(a) \right\} \\
        &= \inf \left\{ L(\exp^\nu \circ \tau) \;\middle|\; \tau : [0, a] \rightarrow \mathcal{U}, \, \tau(0) = \tilde{\zeta}_{ij}(0), \, \tau(a) = \tilde{\zeta}_{ij}(a) \right\} \\
        &\ge \inf \left\{ L(\upsilon) \;\middle|\; \upsilon : [0, a] \rightarrow M, \, \upsilon(0) = \hat{\zeta}_{ij}(0), \, \upsilon(a) = \hat{\zeta}_{ij}(a) \right\} \\
        &= d\left( \hat{\zeta}_{ij}(0), \hat{\zeta}_{ij}(a) \right) = L(\hat{\zeta}_{ij}) = L(\tilde{\zeta}_{ij}).
    \end{align*}
    As argued earlier, the family $\left\{ \tilde{\zeta}_{ij} \right\}_j$ is uniformly Lipschitz, with Lipschitz constant $a_0$. In order to appeal to the Arzel\'{a}-Ascoli theorem, we need to get a uniform bound for these two families. First, we get a $\delta > 0$ sufficiently small, so that the closed forward ball $K_i \coloneqq \overline{B_+(\mathbf{w}_i, 3\delta)}$ is compact and contained in $\mathcal{W}_i$. Since $L\left( \tilde{\zeta}_{ij}|_{[0,\delta]} \right) = L\left( \hat{\zeta}_{ij}|_{[0,\delta]} \right) = \delta \frac{a_{ij}}{a} \rightarrow \delta$, for $j$ large we have $L\left( \tilde{\zeta}_{ij}|_{[0,\delta]} \right) < 2\delta$. As $\tilde{\zeta}_{ij}(0) \rightarrow \mathbf{w}_i$, arguing similarly as in the proof of \autoref{lemma:convergentSubsequenceNSegment}, it then follows that $\tilde{\zeta}_{ij}|_{[0,\delta]} \subset K_i$ for $j$ large. On the other hand, as $\eta$ does not intersect $\mathrm{Cu}(N)$ except at $\eta(0)$, we have $\eta|_{[\delta, a]} \subset \exp^\nu\left( \mathcal{I}(N) \right)$. Since $\hat{\zeta}_{ij}|_{[\delta, a]} \rightarrow \eta|_{[\delta, a]}$ uniformly, we obtain a relatively compact neighborhood $W$ of $\eta|_{[\delta, a]}$, so that $\hat{\zeta}_j|_{[\delta, a]} \subset W \subset \overline{W} \subset \exp^\nu\left( \mathcal{I}(N) \right)$ for $j$ large. Set $K \coloneqq \left( \exp^\nu|_{\mathcal{I}(N)} \right)^{-1}\left( \overline{W} \right)$. Clearly, $\tilde{\zeta}_{ij}|_{[\delta,a]} = \left( \exp^\nu|_{\mathcal{I}(N)} \right)^{-1} \circ \hat{\zeta}_{ij}|_{[\delta, a]} \subset K$. Thus, the family $\left\{ \tilde{\zeta}_{ij} \right\}_j$ is contained in the compact set $K \cup K_i$, for $i = 1,2$. Then, an application of the (asymmetric) Arzel\`a-Ascoli theorem \cite{CoZi07} gives a subsequence of $\left\{ \tilde{\zeta}_{ij} \right\}_j$ converging uniformly to some $\tilde{\zeta}_i \subset \mathcal{I}(N) \cup \mathcal{W}_i$, which joins $\mathbf{w}_i$ to $\tilde{\mathbf{q}}$, respectively.

    Suppose, $\zeta_i \coloneqq \exp^\nu \circ \ \tilde{\zeta}_i $, which then joins $x_0$ to $q$ for $i = 1,2$. Note that 
    \[\zeta_i(t) = \exp^\nu \left( \lim_j \tilde{\zeta}_{ij}(t) \right) = \lim_j \exp^\nu\left( \tilde{\zeta}_{ij}(t) \right) = \lim_j \hat{\zeta}_{ij}(t) = \eta(t), \quad \text{for } 0 \le t \le a.\]
    We now show that this leads to a contradiction. Set 
    \[s_0 := \inf_{0 \le s \le a} \left\{ s \;\middle|\; \tilde{\zeta}_1|_{[s,a]} = \tilde{\zeta}_2|_{[s,a]} \right\}.\]
    Since $\exp^\nu$ is injective on $\mathcal{I}(N)$ and $\tilde{\zeta}_1(a) = \tilde{q} = \tilde{\zeta}_2(a) \in \mathcal{I}(N)$, we have $s_0 < a$. Again, from the hypothesis, $\exp^\nu$ is a local diffeomorphism on $\mathcal{I}(N) \cup \mathcal{W}_1 \cup \mathcal{W}_2$, and hence in particular, injective in a neighborhood of $\tilde{\zeta}_1(s_0) = \tilde{\zeta}_2(s_0)$. Consequently, by a standard argument using the connectivity of the interval $[0, a]$, we get $s_0 = 0$. But this implies that $\tilde{\zeta}_1 = \tilde{\zeta}_2$ everywhere, which contradicts that $\tilde{\zeta}_1(0) = \ell \dot\gamma_1(0) \ne \ell \dot\gamma_2(0) = \tilde{\zeta}_2(0)$.

    Thus, we conclude that there exist exactly two $N$-segments joining $N$ to $x_0$. Suppose, $\gamma_1, \gamma_2 : [0, \ell] \rightarrow M$ are the $N$-segments. We must have one (and exactly one) of $\gamma_i \cup \eta$ is smooth since otherwise arguing as above we will again reach a contradiction. Without loss of generality, assume that $\gamma_1 \cup \eta$ is smooth at $x_0 = \gamma_1(\ell) = \eta(0)$, i.e., $\dot \gamma_1(\ell) = \dot\eta(0)$. But then it follows from the uniqueness of the solution to the geodesic equation that $\gamma_1 \cup \eta$ is an $N$-geodesic. This proves  (\hyperref[lemma:twoNSegments:1]{1}). \medskip

    \noindent\textbf{Case (\hyperref[lemma:twoNSegments:2]{2}):} In order to prove (\hyperref[lemma:twoNSegments:2]{2}), assume that $q \in \mathrm{Cu}(x_0)$. Say, $\eta : [0, a] \rightarrow M$ is a minimizer joining $x_0$ to $q$, where $a = d(x_0, q)$. We can choose some $0 < a_0 < a$, so that $q_0 = \eta(a_0) \not \in \mathrm{Cu}(x_0)$. Consider the function $\mathcal{M}_0 = \mathcal{M}_{q_0}$ defined as $\mathcal{M}_0(x) \coloneqq d(N, x) + d(x, q_0)$ for $x \in M$. We show that $\mathcal{M}_0(x_0) = \min_{x\in\mathrm{Cu}(N)} \mathcal{M}_0(x)$. If not, suppose there exists some $x_1 \in \mathrm{Cu}(N)$ such that $\mathcal{M}_0(x_1) < \mathcal{M}_0(x_0)$. Note that
    \begin{align*}
        \mathcal{M}(x_1) 
        &= d(N, x_1) + d(x_1, q) \\
        &\le d(N, x_1) + d(x_1, q_0) + d(q_0, q) \\
        &= \mathcal{M}_0(x_1) + d(q_0, q) \\
        &< \mathcal{M}_0(x_0) + d(q_0, q) \\
        &= d(N, x_0) + \underbrace{d(x_0, q_0) + d(q_0, q)}_{\text{$x_0, q_0,q$ are on the minimizer $\eta$}} \\
        &= d(N,x_0) + d(x_0, q) \\
        &= \mathcal{M}(x_0),
    \end{align*}
    which is a contradiction as $\mathcal{M}$ achieves a global minima value on $\mathrm{Cu}(N)$ at $x_0$. Hence, $\mathcal{M}_0$ also achieves a global minima on $\mathrm{Cu}(N)$ at $x_0$. But now we can argue as in (\hyperref[lemma:twoNSegments:1]{1}), for $q_0$ and $\mathcal{M}_0$ instead of $q$ and $\mathcal{M}$ respectively. In particular, we have exactly two $N$-segments joining $N$ to $x_0$, say, $\gamma_1, \gamma_2$.
    
    Now, we have a minimizer joining $x_0$ to $q$. Suppose we have at least three distinct minimizers from $x_0$ to $q$. Then, we have at least one geodesic, say, $\eta$ from $x_0$ to $q$ such that both $\gamma_1 \cup \eta$ and $\gamma_2 \cup \eta$ are broken at $x_0$. Let us pick some $q_0$ on $\eta$ with $q_0 \ne x_0, q$, so that $q_0 \not \in \mathrm{Cu}(x_0)$. Denote $\tilde{\eta}$ as the subsegment of $\eta$ joining $x_0$ to $q_0$, and note that $\gamma_1 \cup \tilde{\eta}$ and $\gamma_2 \cup \tilde{\eta}$ are broken at $x_0$. Arguing similarly as in (\hyperref[lemma:twoNSegments:1]{1}), with $\mathcal{M}$ and $q$ replaced by $\mathcal{M}_{q_0}$ and $q_0$ respectively, we again reach a contradiction. Hence, there exist at most two minimizers, say, $\eta_1, \eta_2$, from $x_0$ to $q$. If $\eta_1, \eta_2$ are distinct, then the same argument gives that two (and hence, exactly two) of the curves $\left\{ \gamma_i \cup \eta_j, \; i, j = 1, 2 \right\}$ are smooth, which in turn gives two $N$-geodesics, both crossing $\mathrm{Cu}(N)$ at $x_0$. The other two curves are broken at $x_0$. If $\eta_1 = \eta_2 = \eta$, then we are in the same situation as (\hyperref[lemma:twoNSegments:1]{1}), and the conclusion follows. This completes the proof of (\hyperref[lemma:twoNSegments:2]{2}).
    \begin{figure}[H]
        \centering
    \def\svgwidth{0.5\columnwidth}
    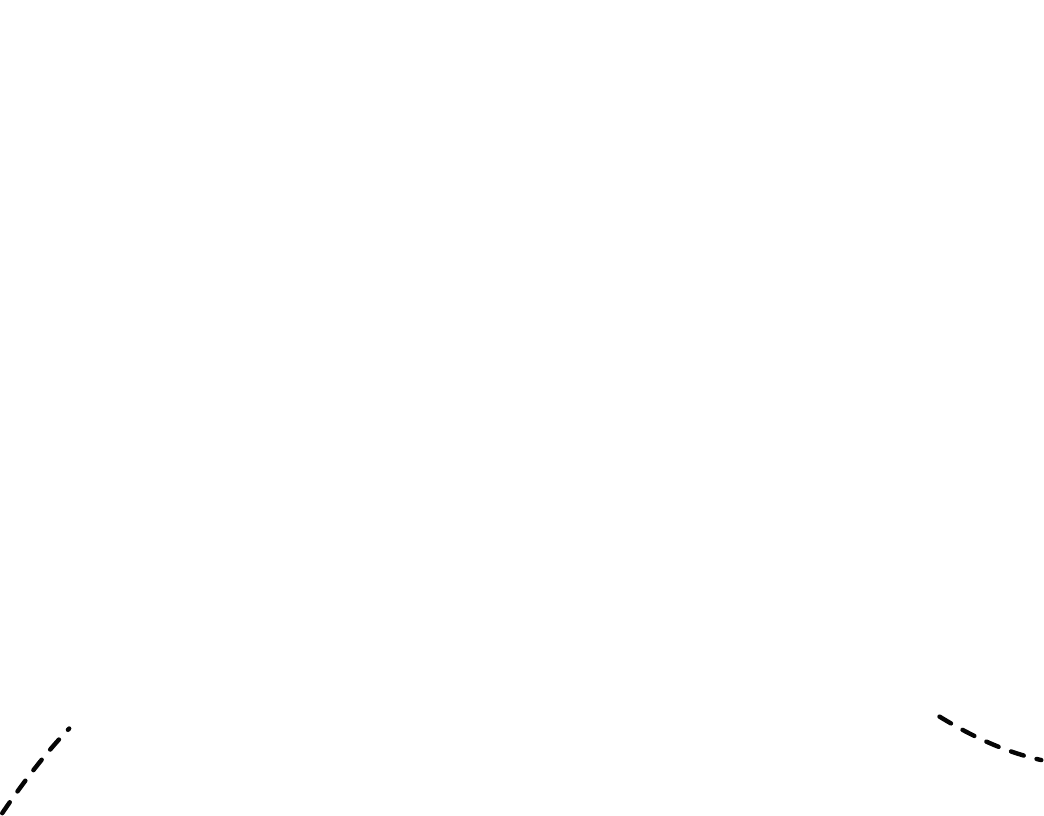

        \caption{The possible $N$-geodesics joining $N$ to $q$}
        \label{fig:reflects-case-2_1}
    \end{figure}

    \noindent \textbf{Case (\hyperref[lemma:twoNSegments:2prime]{2'}):} Lastly, let us assume that $F$ is a reversible Finsler metric so that the distance is now symmetric. In particular, $q \in \mathrm{Cu}(x_0) \Rightarrow x_0 \in \mathrm{Cu}(q)$. If possible, let $\eta$ be the unique minimizer joining $x_0$ to $q$, and without loss of generality, $\gamma_1 \cup \eta$ be smooth at $x_0$. Then, for any point $y \ne x_0$ on $\gamma_1$, the minimizer joining $y$ to $q$ cannot be contained in $\gamma_1 \cup \eta$, otherwise this would extend the reversed minimizer $\bar{\eta}$, contradicting $x_0 \in \mathrm{Cu}(q)$. Consequently, $d(y, q) < d(y, x_0) + d(x_0, q)$. Since $\gamma_2 \cup \eta$ is broken at $x_0$, for any $y \ne x_0$ on $\gamma_2$, we again have $d(y, q) < d(y, x_0) + d(x_0, q)$. But now we are in the same situation as in case (\hyperref[lemma:twoNSegments:1]{1}), and arguing similarly we reach a contradiction. Thus, $\eta$ cannot be the unique minimizer joining $x_0$ to $q$. Arguing as in (\hyperref[lemma:twoNSegments:2]{2}), we conclude that there are exactly two such minimizers, and this finishes the proof of (\hyperref[lemma:twoNSegments:2prime]{2'}).
\end{proof}

We can now prove \autoref{thm:atLeastTwoNGeodesics}.

\begin{proof}[Proof of \autoref{thm:atLeastTwoNGeodesics}]
    Suppose $\mathrm{Cu}(N)$ has no focal points of $N$ so that we can appeal to \autoref{lemma:twoNSegments}. If $q\in \mathrm{Cu}(N)$, then by \autoref{thm:cutPointClassification}, there exist at least two $N$-geodesics (in fact, $N$-segments) joining $N$ to $q$. Suppose $q \not \in \mathrm{Cu}(N)$, so that we have a (unique) $N$-segment joining $N$ to $q$, which clearly does not intersect $\mathrm{Cu}(N)$. Note that if $q \in N$, we consider the constant curve at $q$ as an $N$-segment.
    
    In order to apply \autoref{lemma:twoNSegments}, we need to show that $\mathcal{M}|_{\mathrm{Cu}(N)}$ attains a global minima at some $x_0 \in \mathrm{Cu}(N)$, where $\mathcal{M} = \mathcal{M} _q$ as defined earlier. Set $r = \inf_{x \in \mathrm{Cu}(N)} \mathcal{M}(x)$. We have infinitely many $x_n \in \mathrm{Cu}(N)$ such that $\lim \mathcal{M}(x_n) = r$. Without loss of generality, we assume $r \le \mathcal{M}(x_n) < r + \epsilon$ for some $\epsilon > 0$. We consider the following cases.
    \begin{enumerate}[\quad(a)]
        \item Suppose $N$ is compact. Consider he set $\overline{B_+(N, r + \epsilon)} \coloneqq \left\{ x \in M \;\middle|\; d(N, x) \le r + \epsilon \right\}$, which is compact as $F$ is forward complete. Set $X = \mathrm{Cu}(N) \cap \overline{B_{+}(N, r + \epsilon)} $. As $\mathrm{Cu}(N)$ is closed (\autoref{thm:separatingSetDense}), we have $X$ is a compact set. Since $d(N, x_n) \le \mathcal{M}(x_n) < r+ \epsilon$, we get $x_n \in X$. Consequently, we have a subsequence $x_{n_k} \rightarrow x_0 \in X \subset \mathrm{Cu}(N)$. Clearly, $\mathcal{M}(x_0) = \lim \mathcal{M}(x_{n_k}) = r$, i.e., $\mathcal{M}|_{\mathrm{Cu}(N)}$ attains a global minima at $x_0$
        \item If $(M, F)$ is backward complete, we then have that the backward closed ball $\overline{B_{-}(q, r + \epsilon)}$ is compact, and then similar argument as in (a) works.
    \end{enumerate}
    In any case, from \autoref{lemma:twoNSegments} we now get at least one (and at most two) $N$-geodesic joining $N$ to $q$, which crosses $\mathrm{Cu}(N)$, and thus distinct from the $N$-segment we already have (\autoref{fig:reflects-case-2_1}). This concludes the claim.
\end{proof}

In order to prove the second part of \autoref{thm:NGeodesicLoop}, following \cite{Xu2015}, we get a generalization of \autoref{lemma:twoNSegments}.
\begin{lemma}\label{lemma:twoNGeodesicsLocalMinima}
    Let $N$ be a closed submanifold of a forward complete $(M, F)$, and condition (\hyperref[eq:hypothesisH]{H}) holds. Let $x_0$ be a \emph{local} minima of the function $\mathcal{M}_q|_{\mathrm{Cu}(N)}$, for some $q \in M \setminus \mathrm{Cu}(N)$. Assume furthermore that $\mathrm{Cu}(N)$ is compact. If $x_0$ is not a focal point of $N$, then $(1)$ and $(2)$ of \autoref{lemma:twoNSegments} remain valid.
\end{lemma}
\begin{proof}
    Let $\eta : [0,a] \rightarrow M$ be any minimizer joining $x_0$ to $q$, where $a = d(x_0, q)$. For any $z = \eta(t)$, we have 
    \[\mathcal{M}_{z}(x_0) = d(N, x_0) + d(x_0, z) = d(N,x_0) + d(x_0, q) - d(z, q) = \mathcal{M}_q(x_0) - d(z, q).\]
    Then, for $x \in \mathrm{Cu}(N)$ sufficiently close to $x_0$ we see,
    \begin{align*}
        \mathcal{M}_{z}(x) &=
        d(N, x) + d(x, z) \\
        &\ge d(N,x) + d(x,q) - d(z, q) \\
        &= \mathcal{M}_q(x) - d(z, q) \\
        &\ge \mathcal{M}_q(x_0) - d(z, q) \\
        &= \mathcal{M}_{z}(x_0).
    \end{align*}
    In particular, $x_0$ is again a local minima of $\mathcal{M}_{z}|_{\mathrm{Cu}(N)}$. We show that for some $z = \eta(t)$, $\mathcal{M}_z|_{\mathrm{Cu}(N)}$ attains a global minima at $x_0$. In fact, it suffices to show that $\mathcal{M}_z|_{\mathrm{Cu}(N)}$ attains its global minima sufficiently close to $x_0$ for some $z$, whence, $x_0$ being a local minima will imply it is a global minima as well.
    
    Fix some $t_j \rightarrow 0$ and denote $q_j = \eta(t_j)$, so that $q_j \rightarrow x_0$. If possible, suppose, for some open neighborhood $U \subset \mathrm{Cu}(N)$ of $x_0$, we have that $\mathcal{M}_{q_j}|_{\mathrm{Cu}(N)}$ attains a global minima, say, $z_j \in \mathrm{Cu}(N) \setminus U$. As $\mathrm{Cu}(N)$ is compact by hypothesis, passing to a subsequence, we assume that $z_j \rightarrow z_0 \in \mathrm{Cu}(N) \setminus U$. Now, for any $y \in \mathrm{Cu}(N)$ we have 
    \[d(N, z_j) + d(z_j, q_j) = \mathcal{M}_{q_j}(z_j) \le \mathcal{M}_{q_j}(y) = d(N, q_j) + d(q_j, y).\]
    Taking $j \rightarrow \infty$ we have, 
    \[d(N, z_0) + d(z_0, x_0) \le d(N, x_0) + d(x_0, y) \Rightarrow \mathcal{M}_{x_0}(z_0) \le \mathcal{M}_{x_0}(y), \quad \forall y \in \mathrm{Cu}(N).\]
    In particular, taking $y = x_0$ we get 
    \begin{align*}
        &\mathcal{M}_{x_0}(z_0) \le \mathcal{M}_{x_0}(x_0) \\
        \Rightarrow & d(N, z_0) + d(z_0, x_0) \le d(N,x_0) + 0 \le d(N,z_0) + d(z_0, x_0) \\
        \Rightarrow & d(N, x_0) = d(N,z_0) + d(z_0, x_0).
    \end{align*}
    In particular, $z_0$ must be on an $N$-segment joining $N$ to $x_0$. But $z_0 \in \mathrm{Cu}(N)$, and hence any $N$-segment connecting $N$ to $z_0$ cannot be an $N$-segment after $z_0$. In particular $z_0 = x_0$, which contradicts $z_0 \not \in U$. Thus, we must have for some $q_{j_0} = \eta(t_{j_0})$, $\mathcal{M}_{q_{j_0}}|_{\mathrm{Cu}(N)}$ attains a global minima at $x_0$. Observe that $q_{j_0} \not \in \mathrm{Cu}(x_0)$, since $\eta|_{[0, t_{j_0}]}$ clearly extends to a minimizer $\eta$. Now, applying \autoref{lemma:twoNSegments} (\hyperref[lemma:twoNSegments:1]{1}) for $\mathcal{M}_{q_{j_0}}$ and $q_{j_0}$, we have exactly two $N$-segment, say, $\gamma_1, \gamma_2$ joining $N$ to $x_0$. Furthermore, without loss of generality, only $\gamma_1 \cup \eta|_{[0, t_{j_0}]}$ is smooth at $x_0$. Clearly, $\gamma_1 \cup \eta$ is then an $N$-geodesic joining $N$ to $q$, crossing $\mathrm{Cu}(N)$ at $x_0$, and $\gamma_2 \cup \eta$ is broken at $x_0$. Let us now consider the two cases.

    \begin{itemize}
        \item \textbf{Case (\hyperref[lemma:twoNSegments:1]{1}):} Suppose $q \not \in \mathrm{Cu}(x_0)$. Then, $\eta$ is the unique minimizer joining $q$ to $x_0$ and the proof is done. 
        \item \textbf{Case (\hyperref[lemma:twoNSegments:2]{2}):} Since for any minimizer $\eta$ joining $q$ to $x_0$, precisely one of $\left\{ \gamma_1 \cup \eta, \gamma_2 \cup \eta \right\}$ can be smooth, it is clear that there are at most two minimizers joining $q$ to $x_0$. The claim is then immediate.
    \end{itemize}
    This concludes the proof.
\end{proof}

We now prove \autoref{thm:NGeodesicLoop}.
\begin{proof}[Proof of \autoref{thm:NGeodesicLoop}]
    Suppose $x_0$ is not a focal point of $N$. Pick some $q \in N$ so that $d(q, x_0) = d(N, x_0)$ (\autoref{prop:existenceOfNSegements}). Since $F$ is reversible, we have \[\mathcal{M}_q(x_0) = d(N,x_0) + d(x_0,q) = d(N,x_0) + d(q,x_0) = 2d(N,x_0).\]
    Then, for $x\in \mathrm{Cu}(N)$, we have
    \begin{align*}
        \mathcal{M}_q(x) = d(N, x) + d(x,q) &= d(N,x) + d(q,x) \\
        &\ge d(N, x) + d(N, x), \quad \text{as $q \in N$} \\
        &= 2d(N,x) \\
        &\ge 2d(N,x_0) = \mathcal{M}_q(x_0),
    \end{align*}
    i.e., $x_0$ is a global minima of $\mathcal{M}_q|_{\mathrm{Cu}(N)}$. By \autoref{lemma:twoNSegments}, we then have precisely two $N$-segments, say, $\gamma_1, \gamma_2$ joining $N$ to $x_0$, and at most two minimizers, say, $\eta_1, \eta_2$ joining $x_0$ to $q$. But reversibility of $F$ implies that the reversed curves $\bar{\eta}_1, \bar{\eta}_2$ are minimizers joining $q$ to $x_0$. As $d(N, x_0) = d(q, x_0)$ by choice, we get $\bar{\eta}_1, \bar{\eta}_2$ are $N$-segments joining $N$ to $x_0$. Consequently, without loss of generality, we must have either $\bar{\eta}_1 = \gamma_1, \bar{\eta}_2 = \gamma_2$, or $\gamma_1 = \bar{\eta}_1 = \bar{\eta}_2$. In any case, we furthermore have that $\gamma_1 \cup \bar{\gamma}_2$ is the unique smooth $N$-geodesic, crossing $\mathrm{Cu}(N)$ at $x_0$, which is the midpoint. This $N$-geodesic starts and terminates in $q \in N$, with initial and terminal velocity in $S(\nu_q)$, i.e, $\gamma$ is an $N$-geodesic loop.

    If $x_0$ is a local minima, then observe that the same argument as above gives $\mathcal{M}_q(x) \ge \mathcal{M}_q(x_0)$ for $x\in\mathrm{Cu}(N)$ sufficiently close to $x_0$. We then conclude the proof by using \autoref{lemma:twoNGeodesicsLocalMinima}.
\end{proof}

The following result can be compared to \cite[Lemma 5.2]{InnItoNagShi19}.
\begin{corollary}\label{cor:sphereTheorem}
    Suppose $(M, F)$ is a reversible Finsler manifold of dimension $n \ge 2$. Let $N = \left\{ p, q \right\} \subset M$ be a two point set. Assume the following.
    \begin{itemize}
        \item $\mathrm{Cu}(N)$ contains no focal point of $N$.
        \item For every $x \in \mathrm{Cu}(N)$ we have $d(N, \mathrm{Cu}(N)) = d(N, x)$.
    \end{itemize}
    Then, $M$ is homeomorphic to $S^n$, and $\mathrm{Cu}(N)$ is diffeomorphic to $S^{n-1}$.
\end{corollary}
\begin{proof}
    Suppose $L_0 = d(N, \mathrm{Cu}(N))$. It follows from \autoref{thm:NGeodesicLoop} that for each point $x \in \mathrm{Cu}(N)$, there exists a unique $N$-geodesic loop, say $\gamma_x$, which crosses $\mathrm{Cu}(N)$ at $x$ at the midpoint. By the hypothesis, $\frac{1}{2} L(\gamma_x) = d(N, x) = d(N, \mathrm{Cu}(N)) = L_0 \Rightarrow L(\gamma_x) = 2L_0$. Let us now consider the following three sets.
    \begin{enumerate}[(i)]
        \item $A_{p,q} = \left\{ x \in \mathrm{Cu}(N) \;\middle|\; \begin{alignedat}{2} & \text{ There is a geodesic of length $2L_0$} \\ & \text{ joining $p$ to $q$ with $x$ as the midpoint }  \end{alignedat}  \right\}$,
        \item $A_{p,p} = \left\{ x \in \mathrm{Cu}(N) \;\middle|\; \begin{alignedat}{2} & \text{ There is a geodesic of length $2L_0$} \\ & \text{ joining $p$ to itself with $x$ as the midpoint}  \end{alignedat}  \right\}$,
        \item $A_{q,q} = \left\{ x \in \mathrm{Cu}(N) \;\middle|\; \begin{alignedat}{2} & \text{ There is a geodesic of length $2L_0$} \\ & \text{ joining $q$ to itself with $x$ as the midpoint}  \end{alignedat}  \right\}$.
    \end{enumerate}

    By the above argument, we have 
    \[\mathrm{Cu}(N) = A_{p,q} \sqcup A_{p,p} \sqcup A_{q,q},\]
    as any such geodesics considered in the definitions of the subsets are in fact $N$-geodesic loops. Since $\mathrm{Cu}(N)$ is closed (\autoref{thm:separatingSetDense}), by an application of \autoref{prop:existenceOfNSegements}, all three subsets are closed in $\mathrm{Cu}(N)$. As $\mathrm{Cu}(N)$ is the disjoint union of three closed subsets, all three subsets are open in $\mathrm{Cu}(N)$ as well. Note that $A_{p, q} \ne \emptyset$, as we have a minimizer joining $p$ to $q$, whose midpoint is then in $\mathrm{Cu}(N)$, and clearly, $d(p, q) = 2L_0$. Since $n = \dim M \ge 2$, we have $M \setminus N$ is connected, which implies that $\mathrm{Cu}(N)$ is connected (\autoref{thm:deformation} (\hyperref[thm:deformation:2]{2})). Hence, we have $\mathrm{Cu}(N) = A_{p,q}$. In other words, for each $x \in \mathrm{Cu}(N)$, there exists a unique minimizer $\gamma_x$ of length $2L_0$, joining $p$ to $q$. We have 
    \[\gamma_x(t) = \exp^\nu(t \mathbf{v}) = \exp_p(t \mathbf{v}), \quad \mathbf{v} = \left( \exp_p \right)^{-1}(x).\]
    Now, every geodesic emanating from $p$ is also an $N$-geodesic, and must cross $\mathrm{Cu}(N)$ at some $x \in \mathrm{Cu}(N)$. In particular, all geodesics emanating from $p$ are minimizers, crossing $\mathrm{Cu}(N)$ at the midpoint at length $L_0$, and terminating at $q$ at length $2L_0$. Thus, $q \in \mathrm{Cu}(p)$. On the other hand, for any $z \in M \setminus \left\{ p, q \right\}$, a minimizer joining $p$ to $z$ is extensible to distance $2L_0$ whence it reaches $q$, and so $z \not \in \mathrm{Cu}(p)$. Thus, $\mathrm{Cu}(p) = \left\{ q \right\}$ (and by reversibility $\mathrm{Cu}(q) = \left\{ p \right\}$ as well). Consequently, $M$ is homeomorphic to the sphere $S^n$, since $\exp _p$ identifies the one point compactification of $\mathcal{I} (p)$ with $M$ (\autoref{obs:compactnessAndTangentCutLocus} (\ref{obs:compactness:MIsImage})). We also have, $\mathrm{Cu}(N) = \left\{ \exp_p(L_0 \mathbf{v}) \;\middle|\; \mathbf{v} \in S(T_p M) \right\}$. As $\exp_p$ is a diffeomorphism before the cut time for $p$, which is the constant $2L_0$, we have $\mathrm{Cu}(N)$ is diffeomorphic to $S^{n-1} \cong S(T_p M)$. This concludes the proof.
\end{proof}

\section*{Acknowledgments}
The authors would like to express their gratitude to S. Basu and J. Itoh for many fruitful discussions and suggestions in the preparation of this article. They are also grateful to the anonymous referee for their valuable suggestions which, in particular, helped in removing the compactness assumption in most of the results of the article. They would like to thank ICTS for the wonderful hosting and a research-friendly environment during the workshop Dualities in Topology and Geometry (code: 1139 ICTS/dta2023/05), where the majority of this work was done. The first author was supported by the NBHM grant no. 0204/1(5)/2022/R\&D-II/5649 and the second author was supported by ICTS postdoctoral fellowship. 

\bibliographystyle{alphaurl}

\begin{thebibliography}{Omo68}

    \bibitem[AJ19]{Alves2019}
    Benigno Alves and Miguel~Angel Javaloyes.
    \newblock A note on the existence of tubular neighbourhoods on {F}insler
      manifolds and minimization of orthogonal geodesics to a submanifold.
    \newblock {\em Proceedings of the American Mathematical Society},
      147(1):369--376, 2019.
    \newblock \href {https://doi.org/10.1090/proc/14229}
      {\path{doi:10.1090/proc/14229}}.
    
    \bibitem[BCS00]{Bao2000}
    D.~Bao, S.-S. Chern, and Z.~Shen.
    \newblock {\em An introduction to {R}iemann-{F}insler geometry}, volume 200 of
      {\em Graduate Texts in Mathematics}.
    \newblock Springer-Verlag, New York, 2000.
    \newblock \href {https://doi.org/10.1007/978-1-4612-1268-3}
      {\path{doi:10.1007/978-1-4612-1268-3}}.
    
    \bibitem[BP23]{BaPr21}
    Somnath Basu and Sachchidanand Prasad.
    \newblock A connection between cut locus, {T}hom space and {M}orse--{B}ott
      functions.
    \newblock {\em Algebr. Geom. Topol.}, 23(9):4185--4233, 2023.
    \newblock \href {https://doi.org/10.2140/agt.2023.23.4185}
      {\path{doi:10.2140/agt.2023.23.4185}}.
    
    \bibitem[Buc77]{Buc77}
    Michael~A. Buchner.
    \newblock Simplicial structure of the real analytic cut locus.
    \newblock {\em Proc. Amer. Math. Soc.}, 64(1):118--121, 1977.
    \newblock \href {https://doi.org/10.2307/2040994} {\path{doi:10.2307/2040994}}.
    
    \bibitem[Bus55]{Bus55}
    Herbert Busemann.
    \newblock {\em The geometry of geodesics}.
    \newblock Academic Press Inc., New York, N. Y., 1955.
    
    \bibitem[CZ07]{CoZi07}
    Julia Collins and Johannes Zimmer.
    \newblock An asymmetric {A}rzel\`a-{A}scoli theorem.
    \newblock {\em Topology Appl.}, 154(11):2312--2322, 2007.
    \newblock \href {https://doi.org/10.1016/j.topol.2007.03.006}
      {\path{doi:10.1016/j.topol.2007.03.006}}.
    
    \bibitem[Fin51]{Fin51}
    P.~Finsler.
    \newblock {\em \"{U}ber {K}urven und {F}l\"{a}chen in allgemeinen
      {R}\"{a}umen}.
    \newblock Verlag Birkh\"{a}user, Basel, 1951.
    \newblock \href {https://doi.org/10.1007/978-3-0348-4144-3}
      {\path{doi:10.1007/978-3-0348-4144-3}}.
    
    \bibitem[Has73]{Hass73}
    Badie T.~M. Hassan.
    \newblock The cut locus of a {F}insler manifold.
    \newblock {\em Atti Accad. Naz. Lincei Rend. Cl. Sci. Fis. Mat. Nat. (8)},
      54:739--744 (1974), 1973.
    
    \bibitem[IINS19]{InnItoNagShi19}
    Nobuhiro Innami, Yoe Itokawa, Tetsuya Nagano, and Katsuhiro Shiohama.
    \newblock Blaschke {F}insler manifolds and actions of projective {R}anders
      changes on cut loci.
    \newblock {\em Trans. Amer. Math. Soc.}, 371(10):7433--7450, 2019.
    \newblock \href {https://doi.org/10.1090/tran/7603}
      {\path{doi:10.1090/tran/7603}}.
    
    \bibitem[ISS12]{Innami2012}
    Nobuhiro Innami, Katsuhiro Shiohama, and Toshiro Soga.
    \newblock The cut loci, conjugate loci and poles in a complete {R}iemannian
      manifold.
    \newblock {\em Geometric and Functional Analysis}, 22(5):1400--1406, 2012.
    \newblock \href {https://doi.org/10.1007/s00039-012-0177-4}
      {\path{doi:10.1007/s00039-012-0177-4}}.
    
    \bibitem[Jav14a]{Javaloyes2014}
    Miguel~Angel Javaloyes.
    \newblock Chern connection of a pseudo-{F}insler metric as a family of affine
      connections.
    \newblock {\em Publicationes Mathematicae Debrecen}, 84(1-2):29--43, 2014.
    \newblock \href {https://doi.org/10.5486/PMD.2014.5823}
      {\path{doi:10.5486/PMD.2014.5823}}.
    
    \bibitem[Jav14b]{Javaloyes2014_Correction}
    Miguel~Angel Javaloyes.
    \newblock Corrigendum to ``{C}hern connection of a pseudo-{F}insler metric as a
      family of affine connections'' [mr3194771].
    \newblock {\em Publ. Math. Debrecen}, 85(3-4):481--487, 2014.
    \newblock \href {https://doi.org/10.5486/PMD.2014.7061}
      {\path{doi:10.5486/PMD.2014.7061}}.
    
    \bibitem[Jav20]{Javaloyes2020}
    Miguel~\'{A}ngel Javaloyes.
    \newblock Curvature computations in {F}insler geometry using a distinguished
      class of anisotropic connections.
    \newblock {\em Mediterr. J. Math.}, 17(4):Paper No. 123, 21, 2020.
    \newblock \href {https://doi.org/10.1007/s00009-020-01560-0}
      {\path{doi:10.1007/s00009-020-01560-0}}.
    
    \bibitem[JS15]{Javaloyes2015}
    Miguel~Angel Javaloyes and Bruno~Learth Soares.
    \newblock Geodesics and {J}acobi fields of pseudo-{F}insler manifolds.
    \newblock {\em Publicationes Mathematicae Debrecen}, 87(1-2):57--78, 2015.
    \newblock \href {https://doi.org/10.5486/PMD.2015.7028}
      {\path{doi:10.5486/PMD.2015.7028}}.
    
    \bibitem[JS21]{Javaloyes2021}
    Miguel~Angel Javaloyes and Bruno~Learth Soares.
    \newblock Anisotropic conformal invariance of lightlike geodesics in
      pseudo-{F}insler manifolds.
    \newblock {\em Classical Quantum Gravity}, 38(2):Paper No. 025002, 16, 2021.
    \newblock \href {https://doi.org/10.1088/1361-6382/abc225}
      {\path{doi:10.1088/1361-6382/abc225}}.
    
    \bibitem[Kli59]{Kli59}
    W.~Klingenberg.
    \newblock Contributions to {R}iemannian geometry in the large.
    \newblock {\em Ann. of Math. (2)}, 69:654--666, 1959.
    \newblock \href {https://doi.org/10.2307/1970029} {\path{doi:10.2307/1970029}}.
    
    \bibitem[Kob67]{Kob67}
    Shoshichi Kobayashi.
    \newblock On conjugate and cut loci.
    \newblock In {\em Studies in {G}lobal {G}eometry and {A}nalysis}, pages
      96--122. Math. Assoc. Amer. (distributed by Prentice-Hall, Englewood Cliffs,
      N.J.), 1967.
    
    \bibitem[Mye35]{Mye35}
    Sumner~Byron Myers.
    \newblock Connections between differential geometry and topology. {I}. {S}imply
      connected surfaces.
    \newblock {\em Duke Math. J.}, 1(3):376--391, 1935.
    \newblock \href {https://doi.org/10.1215/S0012-7094-35-00126-0}
      {\path{doi:10.1215/S0012-7094-35-00126-0}}.
    
    \bibitem[Mye36]{Mye36}
    Sumner~Byron Myers.
    \newblock Connections between differential geometry and topology {II}. {C}losed
      surfaces.
    \newblock {\em Duke Math. J.}, 2(1):95--102, 1936.
    \newblock \href {https://doi.org/10.1215/S0012-7094-36-00208-9}
      {\path{doi:10.1215/S0012-7094-36-00208-9}}.
    
    \bibitem[Oht21]{Ohta2021}
    Shin-ichi Ohta.
    \newblock {\em Comparison Finsler Geometry}.
    \newblock Springer International Publishing AG, 2021.
    \newblock \href {https://doi.org/10.1007/978-3-030-80650-7}
      {\path{doi:10.1007/978-3-030-80650-7}}.
    
    \bibitem[Omo68]{Omo68}
    Hideki Omori.
    \newblock A class of {R}iemannian metrics on a manifold.
    \newblock {\em J. Differential Geometry}, 2:233--252, 1968.
    \newblock URL: \url{http://projecteuclid.org/euclid.jdg/1214428438}.
    
    \bibitem[Pet06]{Peter06}
    Ioan~Radu Peter.
    \newblock On the {M}orse index theorem where the ends are submanifolds in
      {F}insler geometry.
    \newblock {\em Houston Journal of Mathematics}, 32(4):995--1009, 2006.
    
    \bibitem[Poi05]{Poin05}
    Henri Poincar{\'e}.
    \newblock Sur les lignes g{\'e}od{\'e}siques des surfaces convexes.
    \newblock {\em Trans. Amer. Math. Soc.}, 6(3):237--274, 1905.
    \newblock \href {https://doi.org/10.2307/1986219} {\path{doi:10.2307/1986219}}.
    
    \bibitem[Pra23]{Pra23}
    Sachchidanand Prasad.
    \newblock Cut locus of submanifolds: A geometric and topological viewpoint.
    \newblock {\em arXiv preprint arXiv:2303.14931}, 2023.
    \newblock \href {https://arxiv.org/abs/2303.14931} {\path{arXiv:2303.14931}},
      \href {https://doi.org/10.48550/ARXIV.2303.14931}
      {\path{doi:10.48550/ARXIV.2303.14931}}.
    
    \bibitem[Rad04]{Rademacher2004}
    Hans-Bert Rademacher.
    \newblock A sphere theorem for non-reversible {F}insler metrics.
    \newblock {\em Mathematische Annalen}, 328(3):373--387, 2004.
    \newblock \href {https://doi.org/10.1007/s00208-003-0485-y}
      {\path{doi:10.1007/s00208-003-0485-y}}.
    
    \bibitem[Rau59]{Rau59}
    Harry~Ernest Rauch.
    \newblock {\em Geodesics and curvature in differential geometry in the large},
      volume~1.
    \newblock Yeshiva University, Graduate School of Mathematical Sciences, 1959.
    
    \bibitem[Run59]{Rund59}
    Hanno Rund.
    \newblock {\em The differential geometry of {F}insler spaces.}
    \newblock Springer-Verlag, Berlin-G\"{o}ttingen-Heidelberg,, 1959.
    
    \bibitem[Sak96]{Sak96}
    Takashi Sakai.
    \newblock {\em Riemannian geometry}, volume 149 of {\em Translations of
      Mathematical Monographs}.
    \newblock American Mathematical Society, Providence, RI, 1996.
    \newblock Translated from the 1992 Japanese original by the author.
    
    \bibitem[She01]{Shen01}
    Zhongmin Shen.
    \newblock {\em Differential geometry of spray and {F}insler spaces}.
    \newblock Kluwer Academic Publishers, Dordrecht, 2001.
    \newblock \href {https://doi.org/10.1007/978-94-015-9727-2}
      {\path{doi:10.1007/978-94-015-9727-2}}.
    
    \bibitem[ST16]{SaTa16}
    Sorin~V. Sabau and Minoru Tanaka.
    \newblock The cut locus and distance function from a closed subset of a
      {F}insler manifold.
    \newblock {\em Houston J. Math.}, 42(4):1157--1197, 2016.
    \newblock \href {https://doi.org/10.1177/001316448204200428}
      {\path{doi:10.1177/001316448204200428}}.
    
    \bibitem[vM81]{Man81}
    Hans von Mangoldt.
    \newblock Ueber diejenigen {P}unkte auf positiv gekr\"{u}mmten {F}l\"{a}chen,
      welche die {E}igenschaft haben, dass die von ihnen ausgehenden
      geod\"{a}tischen {L}inien nie aufh\"{o}ren, k\"{u}rzeste {L}inien zu sein.
    \newblock {\em J. Reine Angew. Math.}, 91:23--53, 1881.
    \newblock \href {https://doi.org/10.1515/crll.1881.91.23}
      {\path{doi:10.1515/crll.1881.91.23}}.
    
    \bibitem[War65]{Warner1965}
    Frank~W. Warner.
    \newblock The conjugate locus of a riemannian manifold.
    \newblock {\em American Journal of Mathematics}, 87(3):575, 1965.
    \newblock \href {https://doi.org/10.2307/2373064} {\path{doi:10.2307/2373064}}.
    
    \bibitem[Wei68]{Wein68}
    Alan~D. Weinstein.
    \newblock The cut locus and conjugate locus of a riemannian manifold.
    \newblock {\em Ann. of Math. (2)}, 87:29--41, 1968.
    \newblock \href {https://doi.org/10.2307/1970592} {\path{doi:10.2307/1970592}}.
    
    \bibitem[Whi35]{Whi35}
    J.~H.~C. Whitehead.
    \newblock On the covering of a complete space by the geodesics through a point.
    \newblock {\em Ann. of Math. (2)}, 36(3):679--704, 1935.
    \newblock \href {https://doi.org/10.2307/1968651} {\path{doi:10.2307/1968651}}.
    
    \bibitem[Wol79]{Wol79}
    Franz-Erich Wolter.
    \newblock Distance function and cut loci on a complete {R}iemannian manifold.
    \newblock {\em Arch. Math. (Basel)}, 32(1):92--96, 1979.
    \newblock \href {https://doi.org/10.1007/BF01238473}
      {\path{doi:10.1007/BF01238473}}.
    
    \bibitem[Wu14]{Wu14}
    Bing-Ye Wu.
    \newblock Volume comparison theorem for tubular neighborhoods of submanifolds
      in {F}insler geometry and its applications.
    \newblock {\em Ann. Polon. Math.}, 112(3):267--286, 2014.
    \newblock \href {https://doi.org/10.4064/ap112-3-5}
      {\path{doi:10.4064/ap112-3-5}}.
    
    \bibitem[Xu15]{Xu2015}
    Shicheng Xu.
    \newblock On conjugate points and geodesic loops in a complete riemannian
      manifold.
    \newblock {\em The Journal of Geometric Analysis}, 26(3):2221--2230, 2015.
    \newblock \href {https://doi.org/10.1007/s12220-015-9625-3}
      {\path{doi:10.1007/s12220-015-9625-3}}.
    
    \bibitem[Zha17]{Zhao2017}
    Wei Zhao.
    \newblock A comparison theorem for finsler submanifolds and its applications.
    \newblock 2017.
    \newblock \href {https://arxiv.org/abs/1710.10682} {\path{arXiv:1710.10682}},
      \href {https://doi.org/10.48550/ARXIV.1710.10682}
      {\path{doi:10.48550/ARXIV.1710.10682}}.
    
\end{thebibliography}

\end{document}